\documentclass[11pt]{article}

\usepackage{amsmath}
\usepackage{graphicx}
\usepackage{subcaption}
\usepackage{caption}
\usepackage{geometry}
\usepackage{booktabs} 
\usepackage{lscape}   
\usepackage{multirow} 

\usepackage[T1]{fontenc}
\usepackage[english]{babel}
\usepackage{a4wide}
\usepackage{amssymb, amsmath}
\usepackage{xcolor}
\usepackage[utf8]{inputenc}
\usepackage{enumitem}
\usepackage{amsmath,amsthm}
\usepackage{graphicx}
\DeclareMathOperator*{\argmax}{arg\,max}

\newtheorem{theorem}{Theorem}[section]

\newtheorem{lemma}{Lemma}[section]

\newtheorem{proposition}{Proposition}[section]

\renewcommand{\P}{\mathbb{P}}
\newcommand{\E}{\mathbb{E}}
\newcommand{\R}{\mathbb{R}}
\newcommand{\Prob}{\mathbb{P}}

\usepackage{pdfcomment}

\usepackage{bm}

\usepackage{tikz}
\usetikzlibrary{arrows,automata,positioning}
\usepackage{pgfplots}
\pgfplotsset{compat=1.13}

\newcommand{\bP}{\mathbb{P}}
\newcommand{\cP}{\mathcal{P}}
\newcommand{\cPpar}{\mathcal{P}(\mu,s,\beta,\varphi)}

\newcommand{\bE}{\mathbb{E}}

\usepackage{dsfont}
\usepackage{hyperref}

\usepackage{natbib}
\usepackage{authblk}

\newcommand{\1}{\mathds{1}}

\title{Robust Competitive Ratio for Deterministic\\ Monopoly Pricing}

\author{
Tim S.G. van Eck\footnote{Email: \texttt{t.s.g.vaneck@tilburguniversity.edu}},\
Pieter Kleer
{\rm and}  Johan S.H. van Leeuwaarden
}
\affil{Department of Econometrics and Operations Research, Tilburg University}

\date{\today}

\begin{document}

\maketitle

\begin{abstract}
We study deterministic monopoly pricing under partial knowledge of the market, where the seller has access only to summary statistics of the valuation distribution, such as the mean, dispersion, and maximum value. Using tools from distributionally robust optimization and max-min analysis, we evaluate pricing strategies based on their competitive ratio (CR). We characterize the worst-case market scenario consistent with the available information and provide a complete solution for minimizing the CR. Our analysis also covers optimal pricing under various measures of dispersion, including variance and fractional moments. Interestingly, we find that the worst-case market for CR coincides with that for expected revenue. Using proof techniques tailored to the CR framework, we further examine how dispersion and maximum valuation influence optimal deterministic pricing. These results offer practical guidance for setting robust prices when market information is limited.
\end{abstract}

\medskip
\noindent\textbf{Keywords:}
 Monopoly pricing, maximin analysis, revenue maximization, competitive ratio, distributionally robust optimization.

\bigskip
\noindent\textbf{Acknowledgements.} This research was supported by the Nederlandse Organisatie voor Wetenschappelijk Onderzoek (NWO), Vici grant number 202.068.

\section{Introduction}
Monopoly pricing refers to the strategy employed by a sole provider in a market to maximize profits by setting prices above competitive levels. A monopolist leverages comprehensive market knowledge to understand consumer behavior and demand elasticity, enabling precise predictions on how price changes impact the probability of sale. In our analysis, market knowledge refers to the valuation distribution that maps each price to its corresponding conversion rate,~i.e.,~the percentage of interested consumers who complete a purchase. By setting the price, the monopolist navigates a trade-off between price and conversion rate to maximize revenue. The goal is to identify the price point where marginal revenue equals marginal cost, ensuring maximum profitability; see e.g.~\citep{loertscher2022monopoly,LOERTSCHER2024103081,weber2024monopoly} and references therein.
This market power allows the monopolist to strategically adjust prices to maintain high conversion rates and maximize revenue, unlike firms in competitive markets where prices are dictated by supply and demand dynamics. We focus on the practically relevant situation where the monopolist has only partial market knowledge, in particular when the monopolist's knowledge is restricted to summary statistics, such as the mean and variance of the valuation distribution. The monopolist’s market power is then significantly weakened, and optimal pricing strategies that rely on complete market knowledge are no longer applicable.

\subsection{Problem description}
The fact that only partial instead of full information is available impairs the monopolist's ability to set optimal prices, leading to increased uncertainty and risk in pricing decisions. To address this, the monopolist will first determine the worst-case market, representing the most unfavorable distribution of consumer valuations given the known statistics. Once the worst-case market is identified, the focus will shift to max-min analysis where revenue under the worst-case valuation distribution is maximized. This max-min analysis is an example of Distributionally Robust Optimization {\citep{Scarf1958,rahimian2019distributionally}}, a branch of mathematics which involves making decisions that perform well across a range of possible distributions, ensuring robust performance even when precise distributional details are unknown. This approach allows the monopolist to make informed pricing decisions even with limited information, thereby mitigating the reduction in market power and ensuring optimal financial outcomes despite the uncertainty.

Standard works on Bayesian and robust monopoly pricing often consider expected revenue as the primary metric; see \citep{riley1983optimal, myerson1981optimal,azar2012optimal,carrasco2018optimal}.
Instead, our analysis of monopoly pricing will use the competitive ratio, which provides a relative measure of performance by comparing the monopolist's revenue under partial information to the theoretical maximum revenue achievable with complete information (see \cite{terlizzese2008ief} for an axiomatization). This ratio helps to normalize performance across different market conditions and valuation distributions, offering a clearer assessment of the pricing strategy's effectiveness \citep{chen2022distribution, giannakopoulos2019robust, eren2010monopoly,  wang2024minimax,bahamou2024fast}. The competitive ratio is commonly used in other contexts as well. 
For instance, in online algorithms and worst-case analysis, the competitive ratio is used to evaluate the performance of algorithms when inputs are not fully known in advance, see,~e.g.,~\citep{borodin2005online}. The competitive ratio has been analyzed in many applications, with some examples being supply chain formation \citep{babaioff2005incentive}, online routing optimization \citep{jaillet2008generalized}, (semi-online) supply chain scheduling \citep{averbakh2012semi}, online customer selection in supply chain models \citep{elmachtoub2016supply}, and prophet inequalities \citep{lucier2017economic}.

We focus on fixed pricing as mechanism due to its transparency and simplicity in both marketing and operations, while building trust and long-term customer relationships. Randomized pricing has often been suggested as an alternative in some robust pricing studies \citep{carrasco2018optimal,wang2024minimax}, potentially leading to higher revenues for certain ambiguity sets due to opportunities for price differentiation and experimentation. 
Moreover, several recent studies reinforce the strength of fixed pricing, demonstrating that it can secure a guaranteed percentage of the maximal revenue achievable by optimal mechanisms. For information with one historical price and assuming that the valuation distribution is regular or has a monotone hazard rate, 
\cite{allouah2023optimal} show that deterministic pricing can perform well compared with randomized mechanisms. In a comparable spirit, the power of deterministic pricing as compared to dynamic pricing has been demonstrated for reusable resources \citep{elmachtoub2023power}, inventory management \citep{elmachtoub2023simple}  and  queueing systems \citep{bergquist2023static}.
This makes fixed pricing particularly well-suited for robust pricing with partial market knowledge, as it balances simplicity and tractability with strong performance. Consequently, we adopt fixed pricing as mechanism throughout the paper. 

\subsection{Contributions}
Finding the worst-case market for the competitive ratio is significantly more challenging than for expected revenue. For expected revenue, one can often utilize the primal-dual approach from semi-infinite linear programming to identify the worst-case market \citep{chen2022distribution,roos2019chebyshev}. However, when applied to the competitive ratio, this method becomes more laborious due to the need for a Charnes-Cooper transformation to deal with its relative nature and different structural constraints. To tackle this, we develop a novel proof method that exploits alternative techniques tailored to the unique requirements of the competitive ratio and the information set with summary statistics. Specifically, our approach involves advanced analytical methods and optimization strategies to handle the complexity introduced by the competitive ratio's relative performance measure. 

This paper makes three main contributions to the understanding of monopoly pricing under partial market knowledge. First, we fully resolve the minimization problem for the competitive ratio, identifying the worst-case market when partial knowledge of the valuation distribution includes the mean, dispersion, and maximum value. Special cases of the considered dispersion measures include variance and fractional moments. Second, we provide solutions to the max-min analysis and offer exact characterizations for optimal prices. For variance, these optimal prices are derived in closed form, showing how they vary as functions of the summary statistics. Our results reveal how dispersion affects optimal pricing and offer practical guidance for adjusting prices in response to market uncertainty. The choice of dispersion measure impacts the range of possible worst-case scenarios. Variance generally restricts outliers, while fractional moments between 1 and 2 allow for fatter tails and extreme scenarios. This affects how the adversary can influence market outcomes, and in turn, how the monopolist should set the optimal price.  
Third, we demonstrate that, for the partial information considered, the worst-case market for the competitive ratio is identical to the worst-case market for expected revenue. This result is surprising because it challenges the general belief that the competitive ratio limits the adversary's ability to create worst-case scenarios. Intuitively, this is because negative scenarios also negatively impact the optimal revenue with complete information, thereby incentivizing the adversary to select a scenario that is more moderate.

\subsection{Further literature}
Monopoly pricing with full market information is well understood with the optimal price being easily determined \citep{riley1983optimal, myerson1981optimal}. The seminal work of \cite{myerson1981optimal} shows that there is a unique optimal price under certain regularity assumptions, while \cite{riley1983optimal} show that it is always optimal to choose a deterministic price in the single-item setting.
In practical settings, however, it is often impossible to acquire complete knowledge of the value distribution.
Such monopoly pricing with partial market information has also been an active research topic, spurred by the larger trend of developing pricing and auction mechanisms that do not excessively depend on the specific information structures;
\citep{wilson1985game} and \citep{carroll2019robustness}. \cite{bergemann2008pricing} pioneered this direction, taking absolute regret as performance measure assuming the monopolist only knows the maximum value. \citet{bergemann2011robust} explore a related setting, considering both absolute regret and expected revenue as performance criteria, while assuming the valuation distribution lies within a given neighborhood of a known reference distribution.

Other now classic works studied the max-min problem with expected revenue as the performance metric and focus on mean-maximum valuation and mean-variance knowledge. \cite{azar2012optimal} examined knowledge of the mean and variance, demonstrating that the optimal posted price maximizing expected revenue in the worst case can be expressed as an explicit function that decays with variance. Subsequent works have expanded on this by considering different knowledge sets, such as \cite{carrasco2018optimal} for a finite number of moments, \cite{kos2015selling} for mean and maximum valuation, and \cite{suzdaltsev2020distributionally} for mean, variance, and maximum valuation.
  
For mean-variance information, 
\cite{azar2012optimal} leveraged the optimal deterministic price to determine a relative performance guarantee, by evaluating the ratio of the expected revenue in the optimal deterministic price and the expected revenue with full knowledge. Since \cite{azar2012optimal} treat the numerator and denominator of the ratio as independent, this performance guarantee is a lower bound for the CR. Accommodating this result, \cite{azar2013parametric} provide an upper bound by evaluating the competitive ratio in a specific two-point distribution. 
Recent breakthroughs have closed this gap between the lower and upper bound.  Using different proof methods, \cite{giannakopoulos2019robust} and \cite{chen2022distribution} both solve the max-min problem and obtain the optimal price in closed form. 

We extend the state-of-the-art in \cite{giannakopoulos2019robust} and \cite{chen2022distribution} in two ways. First, instead of variance we consider a general dispersion measure, which can also encompass fractional moments, among other examples. Second, we assume the seller knows that the maximum valuation will not exceed a certain upper bound. Determining the optimal robust pricing strategy involves solving a semi-infinite fractional program, which is generally mathematically challenging. Existing approaches attempt to overcome this by transforming the fractional objective into a linear one and solving the dual problem, as seen in \cite{chen2022distribution} for the mean-variance case. However, this method becomes prohibitively complex, if not intractable, when extending to more general dispersion measures and incorporating maximum valuation knowledge. In contrast, \cite{giannakopoulos2019robust} avoided solving the semi-infinite fractional program for the mean-variance case by using relaxations, bounding techniques, and sharp closed-form bounds for conditional expectations of the valuation distribution. Such bounds are not readily available for ambiguity sets with general dispersion and maximum valuation information as in this paper.
To tackle these mathematical challenges, we develop a new probabilistic approach inspired by \cite{giannakopoulos2019robust}, but adapted to handle implicit forms of key quantities and to accommodate a higher level of generality. 

While addressing the mathematical intricacies introduced by general dispersion and maximum valuation information is valuable in its own right, leading to new fundamental insights into robust pricing, it also significantly alters the optimal robust pricing strategy from a practical perspective. In \cite{giannakopoulos2019robust} and \cite{chen2022distribution}, for mean-variance knowledge with an unbounded maximum valuation, the optimal price solving the max-min problem decays as a function of variance. In contrast, our analysis shows that with a bounded maximum valuation and large variance ranges, optimal prices increase with variance. This stark difference highlights the impact of including maximum valuation. 
\cite{chen2022distribution} also show that the expected revenue objective results in a lower price than the CR objective. We will show that this ordering no longer holds when the knowledge also includes the maximum valuation. 

\subsection{Outline}
The remainder of the paper is structured as follows. Section~\ref{sec:model} introduces the max-min framework for robust monopoly pricing considered in this work. In Section~\ref{sec:fun}, we present two fundamental properties of the worst-case valuation distribution (the min of max-min) with formal proofs provided in Section~\ref{prroooffs}. Section~\ref{sec:rprice} derives the optimal prices (the max of max-min), offering a partially implicit characterization for the general setting and a closed-form solution when dispersion is measured by variance. We reveal several new insights into optimal pricing and performance, underscoring the complex yet intuitive interplay between price, dispersion, and maximum valuation. Finally, Section~\ref{sec:conclusions} concludes the paper and outlines some promising avenues for future research.

\section{Model description}\label{sec:model}
The monopoly pricing literature traditionally assumes the seller knows for each price $p$ the conversion rate $\bP(X \geq p)$, where $X$ is a random variable expressing an arbitrary customer's willingness-to-pay. The seller then sets the price to maximize the revenue $\textup{REV}(p,\bP)=p\bP(X \geq p)$. This leads to the maximal expected revenue $\textup{OPT}(\bP)=\sup_{p>0}\textup{REV}(p,\bP)$. Here shorthand notation $\bP$ is used for the conversion rate function $\bP(X \geq p)$. One could alternatively refer to $\bP$ as the demand function, as multiplying the conversion rate with the number of potential consumers would given the demand for the product. The terminology we prefer and tend to use throughout the paper is 
valuation distribution, as $\bP(X \geq p)$ is the tail cumulative distribution function of the valuation $X$. A consumer purchases the product when the consumer's value equals or exceeds the price. 

Observe that $\textup{REV}(p,\bP)$ and $\textup{OPT}(\bP)$ both depend on $\bP$, which implies that the seller knows $\bP$. Being a distribution function, we assume that $\bP(X\geq 0)=1$ and $\lim_{p\to\infty} \bP(X\geq p)=0$. Guaranteeing uniqueness of the optimal price $\argmax_p \textup{REV}(p,\bP)$ requires additional assumptions on $\bP$, such as a monotone hazard rate or some related notion. However, we do not impose such an assumption, as in our setting, the seller does not know $\bP$, and only has partial knowledge of $\bP$, such as the mean, variance and maximum valuation.
Setting the optimal price then involves decision-making under non-Bayesian uncertainty, where the seller lacks full knowledge of the market. Instead of assuming a probabilistic demand model $\bP$, the seller adopts a worst-case scenario approach, where an adversarial Nature is assumed to select the most unfavorable distribution of outcomes. In this setting, the seller must still optimize over all prices, but the adversarial Nature of the market creates a worst-case performance evaluation, influencing the optimal pricing strategy based on both the chosen performance metric and the amount of market information available to the seller.

Two commonly studied performance metrics in this context are maximin revenue and maximin competitive ratio. The maximin revenue metric focuses on protecting the seller against the lowest possible revenue, safeguarding against markets where demand is weak. Let $\cP$ denote the ambiguity set containing all valuation distributions $\bP$ that comply with the seller's partial knowledge.
The seller's goal is then to maximize the revenue in the worst possible circumstances,~i.e.,
\begin{align}\label{mxxm}
 \sup_{p> 0} \inf_{\bP\in \cP} \textup{REV}(p,\bP)=
 \sup_{p> 0} \textup{REV}(p,\bar\bP)= \textup{OPT}(\bar\bP)
\end{align}
with $\bar\bP$ the worst-case valuation distribution solving $ \inf_{\bP\in \cP} \textup{REV}(p,\bP)$ for a fixed $p$. Finding $\bar\bP$ can be challenging and requires solving a minimization problem over the possibly infinite many distributions $\bP$ contained in the ambiguity set $\cP$. Moreover, $\bar\bP$ can have a complicated structure, as Nature's adversarial choice may depend on the price $p$. 

The maximin competitive ratio aims to strike a balance. Not only offering protection against poor revenue outcomes, but also seeking solid performance in favorable, high-revenue scenarios. These differences in the metrics may lead to distinct robust pricing mechanisms, meaning the optimal price for one criterion could differ substantially from that for another. In more formal terms, the competitive ratio (CR) is defined as
\begin{align}
\textup{CR}(p,\bP)=\frac{\textup{REV}(p,\bP)}{\textup{OPT}(\bP)}
\end{align}
and the optimal price for a seller having partial market knowledge contained in $\cP$ should follow from solving the maximin ratio problem 
\begin{align}\label{CPR}
 \sup_{p> 0} \inf_{\bP\in \cP} \textup{CR}(p,\bP). 
\end{align}
As for the maximin revenue problem \eqref{mxxm} solving \eqref{CPR} is generally challenging with a solution that depends strongly on the ambiguity set $\cP$. Compared to \eqref{mxxm}, the maximin problem \eqref{CPR} is likely even more challenging, as the numerator and the denominator of the ratio $\textup{CR}(p,\bP)$ are both functions of the same valuation function $\bP$, bringing an element of non-linearity into the minimization problem.

This paper's goal is to solve \eqref{CPR} for ambiguity sets that lead to non-degenerate yet tractable pricing rules. We evaluate the performance of a pricing mechanism relative to a valuation distribution $\bP$ by considering the ratio between the revenue generated by the mechanism and the revenue generated by the optimal pricing rule. This ratio lies within the interval $[0,1]$, where a higher ratio indicates that the seller's knowledge of the valuations is more informative and leads to better revenue maximization.
When the valuation distribution $\bP$ is fully known, the optimal pricing mechanism maximizes this ratio, achieving the best possible revenue. However, when the valuation distribution is only partially known—belonging to an ambiguity set $\cP$—the pricing problem becomes more complex. In this case, the seller must consider a combined maximization and minimization problem, where the adversary selects the worst-case market from the ambiguity set $\cP$ in response to the seller's chosen price.
The seller's objective is then to select a single price that maximizes the worst-case ratio of the revenue obtained under distribution-free pricing (based on the partial knowledge of $\bP$) to the revenue that would be achieved if the true distribution $\bP$ were known.
The optimal (maximin) ratio represents a lower bound on the performance of any $\bP$-independent pricing strategy. We refer to the price that solves this optimization problem as the {\it optimal robust price} because it is derived without requiring precise knowledge of the valuation distribution, relying instead on partial information. This robust pricing mechanism provides a safeguard against worst-case scenarios, despite considerable uncertainty about the true distribution of valuations.

Let us now introduce in detail the partial knowledge that the seller has in this paper. Define a class of ambiguity sets by conditioning on the mean, maximum valuation and a general dispersion measure as
\begin{align}  
\cP =\cP(\mu,s,\beta,\varphi) =  \{\Prob :  \bP(X \in [0,\beta])=1,\; \E_{\P}(X) = \mu, \E_{\P}(\varphi(X)) = s\},
    \label{eq:ambiguityk}
\end{align}
where $\mu$ is the mean, $s$ the amount of dispersion, $\beta$ the maximum valuation upper bound, and $\varphi : [0,\beta] \rightarrow \R_{\geq 0}$ a strictly convex differentiable function representing the dispersion measure. The general dispersion measure $\varphi$ includes as a subclass all (fractional) moments by allowing $\varphi(x)=x^q$ with $q>1$. In many applications of distributionally robust optimization, the benchmark ambiguity set contains all distributions with a given mean and variance ($q=2$). This ambiguity set was first considered in 
\cite{Scarf1958}, who solved a maximin version of the newsvendor problem with mean-variance information, and for monopoly pricing this started with \cite{azar2012optimal}. For $q< 2$ and $\beta = \infty$, this allows heavy-tailed distributions with infinite second moments as candidate worst-case scenarios. In this way, we also find a new relationship between the robust price and the tail exponent of the worst-case valuation distribution. A comparable setting with heavy-tailed demand in the newsvendor model is studied in \cite{das2021heavy}.

Taken together, the goal is to solve the maximin ratio problem \eqref{CPR} for the ambiguity set \eqref{eq:ambiguityk}. To do so, we first solve the minimization problem for finding the worst-case valuation distribution. We start with presenting a two fundamental expressions in Section~\ref{sec:fun}, relating the worst-case ratio to tight bounds for the tail probability and conditional expectation of the valuation distribution.

\section{Theoretical results}\label{sec:fun}
In this section, we present the key theoretical advances for solving the minimization components of the maximin ratio problem stated in \eqref{CPR}. The first result (Theorem~\ref{CRDT}) provides a characterization of the worst-case ratio in terms of tight bounds for three fundamental quantities: the worst-case tail bound, the best-case tail bound, and the best-case conditional expectation. Deriving these tight bounds involves solving three independent tractable minimization problems, whose solutions together then solve the original ratio problem in  \eqref{CPR}. 
The new formulation for the worst-case ratio introduced in Section~\ref{mainthr1} is leveraged in Section~\ref{sec:rprice} to solve the maximin ratio problem and determine the optimal pricing scheme. In Section~\ref{mainthr2}, we present a second key result (Theorem~\ref{thm22}), demonstrating that, for the ambiguity set considered in this work, the extremal distribution that delivers the worst-case ratio is identical to the distribution that yields the worst-case tail bound. This finding is both unexpected and significant, as it offers important insights into the development of a robust pricing mechanism. The proofs of the two theorems are presented in Section~\ref{prroooffs}.

\subsection{Worst-case ratio}\label{mainthr1}
We will now present a fundamental result for the inner minimization problem of \eqref{CPR}. 
\begin{theorem}[Competitive ratio decomposition]\label{CRDT}
Consider a fixed $p\in(0,\beta]$ and ambiguity set $\cP = \cP(\mu,s,\beta,\varphi)$.
The tight lower bound for $\textup{CR}(p,\bP)$ satisfies
\begin{align}\label{decomposition}
    \inf_{\bP\in \cP} \textup{CR}(p,\bP) = \min\left\{\frac{\inf_{\bP\in \cP}\bP(X\geq p)}{\sup_{\bP\in \cP}\bP(X\geq p)}, \frac{p}{\sup_{\bP\in \cP}\bE(X|X \geq p)}\right\}.
\end{align}
\end{theorem}

Let us first demonstrate 
Theorem~\ref{CRDT}
for the simpler case where the seller only knows the mean and maximum valuation, and lacks dispersion information. Let $\cP(\mu,\beta)$ denote the ambiguity set containing all distributions with support contained in $[0,\beta]$ and mean $\mu$, i.e. \begin{align}  
\cP(\mu,\beta) = \{\Prob :  \E_{\P}(1) = 1,\; \E_{\P}(X) = \mu, \; 0 \leq X \leq \beta\}.
\label{eq:ambimeanrange}
\end{align}
In this case, the tight bounds for the three key quantities in Theorem~\ref{CRDT}  can be readily determined, and combining these bounds yields a tight bound for the competitive ratio. 
Tight lower and upper bounds exist for the tail probability when $p \in (0,\beta]$, as shown in \cite{de1995general}:
\begin{align}
\inf_{\bP\in \cP(\mu,\beta)}\bP(X\geq p) &=\max\Big\{\frac{\mu-p}{\beta-p},0\Big\},\label{WCTP} \quad 
\sup_{\bP\in \cP(\mu,\beta)}\bP(X\geq p) = \min\Big\{\frac{\mu}{p},1\Big\}.
\end{align}
For the required bound on the conditional expectation, consider the two-point distribution $\bP_2 \in \cP(\mu,\beta)$ defined as
\begin{align}\label{P_2}
    \bP_2 = \left\{ \begin{array}{ll}
        p^- & \text{w.p. } (\beta-\mu) / (\beta - p^-), \\
        \beta & \text{w.p. } (\mu-p^-) / (\beta - p^-),
    \end{array}\right.
\end{align}
for which $\bE(X|X\geq p)$ evaluates to $\beta$ and $p^-$ is the left-limit $\lim_{x\uparrow p} x$ of $p$. Since this is a lower bound for $\sup_{\bP\in\cP(\mu,\beta)}\bE(X|X\geq p)$, which in turn is upper bounded by $\beta$, this shows that 
\begin{align}
    \sup_{\bP\in\cP(\mu,\beta)}\bE(X|X\geq p)=\beta.\label{BCCE}
\end{align}

\noindent Combining the three bounds as in Theorem~\ref{CRDT} then gives
    \begin{align}\label{CRKM}
    \inf_{\bP\in \cP(\mu,\beta)} \textup{CR}(p,\bP) =
    \begin{cases}
    \min\left\{\frac{\mu-p}{\beta-p}, \frac{p}{\beta}\right\}, \quad &  p\in(0,\mu], \\
      0, \quad & p\in[\mu,\beta].
    \end{cases}
    \end{align}
This result  has recently been independently derived by \cite{wang2024power}. The closed-form expression in \eqref{CRKM} allows us to determine explicitly the max-min optimal price.

We will now provide another derivation of \eqref{CRKM} from first principles, without using Theorem~\ref{CRDT}. This alternative derivation will shed light on how Theorem~\ref{CRDT} can be proved later in the paper.
First consider $p > \mu$. Nature can select a degenerate distribution that places all probability mass at $\mu$, resulting in $\inf_{\bP\in \cP(\mu,\beta)} \textup{CR}(p,\bP) = 0$.
For $p \leq \mu$, observe that
\begin{align}
\textup{OPT}(\bP_2) &= \sup_{t>0} t \bP_2(X \geq t) = \max\left\{p^-\bP_2(X \geq p^-), \beta\bP_2(X \geq \beta)\right\} = \max\left\{p, \beta\frac{\mu-p}{\beta-p}\right\},
\end{align}
and 
\begin{align}
    \textup{REV}(p,\bP_2) &= p\bP_2(X\geq p) = p\bP_2(X = \beta)=p\frac{\mu-p^-}{\beta-p^-} = p\frac{\mu-p}{\beta-p}.
\end{align}
Hence,
\begin{align}
    \textup{CR}(p,\bP_2) =
    \frac{\textup{REV}(p,\bP_2)}{\textup{OPT}(\bP_2)} = \min\left\{\frac{\mu-p}{\beta-p}, \frac{p}{\beta}\right\}.
\end{align}
Next, we will show for an arbitrary distribution from the ambiguity set 
$\bP_a \in \cP{(\mu, \beta)}$ that $\textup{CR}(p,\bP_a) \geq \min\{\frac{\mu-p}{\beta-p}, \frac{p}{\beta}\}$. Let $p^*$ denote an optimal take-it-or-leave-it price for $\bP_a$, so that $\textup{OPT}(\bP_a) = p^* \bP_a(X\geq p^*)$. 
We now consider a natural distinction that was also used in \cite{giannakopoulos2019robust} and \cite{chen2019distributionally}, distinguishing between $p^*\leq p$ and  $p^*> p$.
 For $p^* \leq p$ we get
\begin{align}
    \textup{CR}(p;\bP_a) &= \frac{p\bP_a(X\geq p)}{p^*\bP_a(X\geq p^*)} \geq \bP_a(X\geq p) \geq \inf_{\bP \in \cP{(\mu, \beta)}}\bP(X\geq p)=\frac{\mu-p}{\beta-p} \geq \min\left\{\frac{\mu-p}{\beta-p}, \frac{p}{\beta}\right\}.
\end{align}
For $p^* > p$ the following tight lower bound can be constructed:
\begin{align}
    \textup{CR}(p;\bP_a) &= \frac{p\bP_a(X \geq p)}{p^* \bP_a(X \geq p^*)} \geq
    \frac{p}{\beta} \geq \min\left\{\frac{\mu-p}{\beta-p}, \frac{p}{\beta}\right\}.\label{CRMR}
\end{align}
The equality between the lower and upper bound proves the assertion.

The above reasoning for $\cP(\mu,\beta)$ shows that the two-point distribution $\bP_2$ is crucial for establishing the tight CR-bound. For the general ambiguity set $\cP$ in \eqref{eq:ambiguityk}, the problem becomes more challenging, as establishing a tight lower bound requires a delicate interplay between two- and three-point distributions, which depends subtly on the chosen dispersion measure. The path leading to Theorem~\ref{CRDT} begins by identifying, for each of the three key quantities, the extreme distributions within the ambiguity set that yield the sharpest possible bounds. By combining these three extreme distributions, we derive a tight lower bound for the competitive ratio. Together with a matching upper bound this will lead to the proof of Theorem~\ref{CRDT} presented in Section~\ref{prroooffs}.

\subsection{Worst-case distribution}\label{mainthr2}
From the detailed analysis of the special case $\cP(\mu,\beta)$, where the seller knows both the mean and the maximum valuation, it becomes evident that the two-point distribution $\bP_2$ not only achieves a tight bound for the competitive ratio but also minimizes expected revenue. This finding suggests that Nature selects the same worst-case valuation distribution, whether the performance measure is $\textup{CR}(p,\bP)$ or $\textup{REV}(p,\bP)$. Our next main result demonstrates that this insight extends to the more general information setting considered in this paper.
\begin{theorem}[Robust worst-case distribution]\label{thm22}
Consider a fixed $p\in(0,\beta]$ and ambiguity set $\cP = \cP(\mu,s,\beta,\varphi)$.
The worst-case objective values
$$\inf_{\bP \in \cP}\textup{CR}(p,\bP) \text{ and } \inf_{\bP \in \cP}\textup{REV}(p,\bP),$$ are asymptotically attained by the same worst-case limiting distribution.\footnote{We say that a limiting distribution asymptotically attains $\inf_{\bP \in \cP}\textup{CR}(p,\bP)$ if the limiting distribution can be represented as a sequence of distributions $\{\bP_n\}_{n=1}^\infty$ with $\bP_n\in\cP$ for all $n\in\mathbb{N}$ and $\lim_{n\xrightarrow{}\infty}\textup{CR}(p,\bP_n)=\inf_{\bP\in\cP}\textup{CR}(p,\bP)$. The same definition applies to $\inf_{\bP\in\cP}\textup{REV}(p,\bP)$ by replacing $\textup{CR}$ with $\textup{REV}$.}
\end{theorem}

The worst-case distribution in Theorem~\ref{thm22} may depend on $p$ and turns out to be either a two-point or a three-point distribution; the proof of Theorem~\ref{thm22} and an explicit description of the worst-case distribution are provided in  Section~\ref{prroooffs}. 
Theorem~\ref{thm22} establishes a key conceptual insight: when fixed prices are used, and the seller knows the mean, dispersion, and potentially an upper bound on the maximum valuation, Nature's adversarial selection is the same for both expected revenue and the competitive ratio. This suggests that differences in optimal pricing arises not from Nature’s choices but from the seller’s focus on optimizing distinct objective functions.

Theorem~\ref{thm22} shows that the difference between expected revenue and competitive ratio does not alter Nature's fundamental adversarial choice. This insensitivity for the criterion, however, applies only to the minimization phase of the maximin problem determining the robust price. The maximization phase, which takes the respective tight bound from the minimization phase as input, may still yield different optimal prices. In Section~\ref{sec:rprice}  we analyze and compare these optimal prices, demonstrating that the prices for the competitive ratio and expected revenue share similar qualitative characteristics. On the other hand, prices derived from the competitive ratio, when viewed as a function of dispersion, tend to be more balanced compared to those optimized for expected revenue.

Theorem \ref{CRDT} and \ref{mainthr2} both continue to hold when introducing a unit cost $c>0$, changing the inner minimization problem in \eqref{CPR} for fixed $p$ into
\begin{align}\label{CR_c}
    \inf_{\bP\in\cP}\frac{(p-c)\bP(X\geq p)}{\sup_{t>c}(t-c)\bP(X\geq t)}.
\end{align}
To see why, notice that \eqref{CR_c} is equivalent to 
\eqref{CPR}, provided that $\cP$ is replaced with
\begin{align}  
\bar{\cP} = \bar{\cP}(\mu,s,\alpha,\beta,\varphi) =  \{\Prob :  \bP(X \in [\alpha,\beta])=1,\; \E_{\P}(X) = \mu, \E_{\P}(\varphi(X)) = s\}
    \label{eq:ambiguityk2}
\end{align}
with $\alpha$ a (possibly negative) lower bound on the valuation. Indeed, with $\bar{p}=p-c$, $\bar{t} = t - c$, $\bar{X} = X - c$, and $\bar{\cP} = \bar{\cP}(\mu-c,s,-c,\beta-c,\varphi)$,
\begin{align}
    \inf_{\bP\in\cP}\frac{(p-c)\bP(X\geq p)}{\sup_{t>c}(t-c)\bP(X\geq t)} &= \inf_{\bP\in\cP}\frac{\bar{p}\bP(X\geq \bar{p}+c)}{\sup_{\bar{t}>0}\bar{t}\bP(X\geq \bar{t}+c)}=\inf_{\bP\in\bar{\cP}}\frac{\bar{p}\bP(\bar{X}\geq \bar{p})}{\sup_{\bar{t}>0}\bar{t}\bP(\bar{X}\geq \bar{t})}.
\end{align}
This shows that incorporating a unit cost results in a shifted ambiguity set. Since the three key quantities \eqref{eq:fundq} underpinning Theorem \ref{CRDT} and \ref{mainthr2} are invariant under such a shift, the key results in this paper (for $c=0$) carry over to the setting with $c>0$. The remainder of this paper assumes 
$c=0$ for reasons of clarity and exposition. 

\section{Proofs of theoretical results}\label{prroooffs}
In Section~\ref{sec23points}, we first introduce several two- and three-point distributions that belong to the ambiguity set and satisfy the constraints on mean, maximum valuation and dispersion. Then, in Section~\ref{sec23pointsfund}, we demonstrate that specific combinations of these two- and three-point distributions serve as extremal distributions, using semi-infinite linear programming techniques, see, e.g., \cite{popescu2005semidefinite} for an overview. We thereby provide tight bounds for the three key quantities. In Section~\ref{sec23pointsfundproof}, we show how all these elements come together to prove Theorem~\ref{CRDT}.

\subsection{Two- and three-point distributions}\label{sec23points}
{Throughout this work, we will rely on two- or three-point distributions with strictly positive probability mass on at most two or three values in $[0,\beta]$.} Let $\cP_2 = \mathcal{P}_2(\mu,s,\beta,\varphi)$ and $\cP_3 = \mathcal{P}_3(\mu,s,\beta,\varphi)$ denote the sets of two-point distributions and three-point distributions within 
$\cP = \cPpar$, respectively. We will start by explaining in detail the structure of two-point distributions.

\paragraph{Two-point distributions.} If we ignore the maximum valuation upper bound for a moment, then it follows from \cite[Proposition 2.3]{kleer2024distribution} that for every $p \in (0,\mu)$ there exists an $\alpha(p) \in (\mu,\infty)$ and that for every $p \in (\mu,\infty)$ there exists an $\alpha(p) \in [0,\mu)$ such that there is a two-point distribution supported on $\{p,\alpha(p)\}$ that is contained in the ambiguity set 
$$
\mathcal{P}(\mu,s,\varphi) = \{\Prob :  \E_{\P}(\1_{[0,\infty)}(X)) = 1,\; \E_{\P}(X) = \mu,\; \E_{\P}(\varphi(X)) = s\},
$$
and the function $\alpha(p)$ is increasing in $p$. Furthermore, using the three constraints defining the ambiguity set $\mathcal{P}(\mu,s,\varphi)$, it follows from \cite[Proposition 2.2]{kleer2024distribution} that $\alpha(p)$ is the unique solution to the equation
\begin{align}\label{alpha}
    \varphi(\alpha(p))\frac{\mu-p}{\alpha(p)-p}+\varphi(p)\frac{\alpha(p)-\mu}{\alpha(p)-p}=s.
\end{align}
This means the set $\cP_2$ can be parameterized by $p$, using the distributions 
\begin{align}\label{P^*_2}
    \bP^*_2(p) = \left\{ \begin{array}{ll}
        p & \text{w.p. } v_p(p), \\
        \alpha(p) & \text{w.p. } v_{\alpha}(p),
    \end{array}\right.
\end{align}
where the probability masses can be shown to be
\begin{align}
\begin{array}{ll}
 v_p(p) & = \displaystyle \frac{\alpha(p)-\mu}{\alpha(p)-p},\\[3.5ex] 
        v_{\alpha}(p) & = \displaystyle  \frac{\mu-p}{\alpha(p)-p}.
\end{array}
\label{eq:two_point_masses} 
\end{align}
Note that, as $p$ increases, probability mass is shifted from $\alpha(p)$ to $p$.
Given the maximum valuation constraint $0 \leq X \leq \beta$ in $\cPpar$, we have a special interest in the two ``extreme''  scenarios for our ambiguity set in \eqref{eq:ambiguityk}: The two-point distribution with $p = 0$, and the two-point distribution with $\alpha(p) = \beta$.
For $p = 0$, we write $\tau_2 = \alpha(0)$, where $\tau_2$ is now the solution to 
\begin{align}
        \frac{s- \varphi(0)}{\mu} = \frac{\varphi(\tau_2)-\varphi(0)}{\tau_2},
        \label{eq:upsilon_2}
\end{align}
and, hence, the right support point of the two-point distribution with support $\{0,\tau_2\}$. This equation is obtained by rewriting \eqref{alpha}. We remark here that this distribution is only contained in $\cPpar$ if $\mu \leq \tau_2 \leq \beta$; we will come back to this later in Lemma \ref{lemma:non-empty}. The $p$ that satisfies $\alpha(p) = \beta$, we denote by $\tau_1$, which is then the left support point of the distribution supported on $\{\tau_1,\beta\}$. The point $\tau_1$ is the solution to
\begin{align}
    \varphi(\tau_1)\cdot \frac{\beta-\mu}{\beta-\tau_1} + \varphi(\beta)\cdot \frac{\mu - \tau_1}{\beta - \tau_1} = s.
    \label{eq:upsilon_1}
\end{align}
If we would model the variance of a distribution using the dispersion constraint $\mathbb{E}[X^2] = \mu^2 + \sigma^2$, we would have  $\tau_1 = \mu - \sigma^2/(\beta - \mu)$ and $\tau_2 = \mu + \sigma^2/\mu$. 

As we mentioned earlier, the two-point distribution supported on $\{0,\tau_2\}$ is only contained in $\mathcal{P}_2$ if $\mu \leq \tau_2 \leq \beta$. This property, as shown in Lemma~\ref{lemma:non-empty}, characterizes the non-emptiness of $\cP$. The proof is provided in Appendix~\ref{ap:proof_lne}.

\begin{lemma}\label{lemma:non-empty}
    $\cPpar \neq \emptyset$ if and only if $\mu \leq \tau_2 \leq \beta$ with $\tau_2$ the solution to \eqref{eq:upsilon_2}.
\end{lemma}

\noindent From this point on, we consider only the case $\beta > \tau_2$ to avoid trivial instances.

\paragraph{Three-point distributions.} 
For $p \in [\tau_1, \tau_2]$ we are interested in the three-point distribution supported on $\{0,p,\beta\}$. We next claim that such a distribution indeed exists in $\cP$. We let $w_0(p), w_p(p), w_{\beta}(p)$ be the probability mass on the points $0,p$ and $\beta$, respectively. Then the distribution on $\{0,p,\beta\}$ with these probabilities is contained in $\cP$ if the system
\begin{align}
    \left\{
\begin{array}{rrrrrrr}
     w_0(p) & + & w_p(p) & + & w_\beta(p) &= & 1 \\
      &  & p\cdot w_p(p) & + & \beta \cdot w_\beta(p) &= & \mu \\
     \varphi(0) w_0(p) & + & \varphi(p)w_p(p) & + & \varphi(\beta)w_\beta(p) & = & s \\
     w_0(p), &  & w_p(p), &   & w_\beta(p) & \geq & 0 \\
\end{array}\right.
\label{eq:equations_threepoint}
\end{align}
has a feasible solution. Solving the system formed by the three equations gives
\begin{align}
\begin{array}{ll}
 w_0(p) & = \displaystyle \frac{s(\beta-p) + (\mu - \beta)\varphi(p) + (p-\mu)\varphi(\beta)}{\beta(\varphi(0) - \varphi(p)) + p(\varphi(\beta) - \varphi(0))},\\[3.5ex] 
        w_p(p) & = \displaystyle  \frac{\beta(\varphi(0)-s) + \mu(\varphi(\beta) - \varphi(0))}{\beta(\varphi(0) - \varphi(p)) + p(\varphi(\beta) - \varphi(0))},  \\[3.5ex] 
        w_\beta(p) & = \displaystyle  \frac{\mu(\varphi(0) - \varphi(p)) - p(\varphi(0)-s)}{\beta(\varphi(0) - \varphi(p)) + p(\varphi(\beta) - \varphi(0))}.
\end{array}
\label{eq:three_point_masses} 
\end{align}
It remains to show that $w_0(p),w_p(p),w_{\beta}(p) \geq 0$. We establish this in Lemma \ref{lem:nn3pd}, of which the proof is provided in Appendix \ref{ap:proof_nn3pd}.

\begin{lemma}\label{lem:nn3pd}
The quantities $w_0(p),\; w_p(p)$ and $w_\beta(p)$  in \eqref{eq:three_point_masses} are non-negative.    
\end{lemma}

\noindent Hence, the distribution
\begin{align}\label{P^*_3}
    \bP^*_3(p) = \left\{ \begin{array}{ll}
        0 & \text{w.p. } w_0(p), \\
        p & \text{w.p. } w_p(p), \\
        \beta & \text{w.p. } w_\beta(p),
    \end{array}\right.
\end{align}
with the weights $w_0(p),w_p(p),w_\beta(p)$ as in \eqref{eq:three_point_masses} is a well-defined probability distribution in $\cPpar$.

\color{black}
\subsection{Key probabilistic quantities}\label{sec23pointsfund}
In this section, we study the quantities
\begin{align}\label{eq:fundq} 
    \sup_{\bP\in\cP}\bP(X\geq p),\; \inf_{\bP\in\cP}\bP(X\geq p) \; \text{ and } \; \sup_{\bP\in \cP}\bE(X|X \geq p),
\end{align}
and show how $\bP^*_2(p)$ and $\bP^*_3(p)$ relate to them. We occasionally omit the function brackets for convenience and write $\bP^*_2$ and $\bP^*_3$ instead.

\begin{proposition}\label{prop:sup_tp}
    Consider ambiguity set $\cP = \cP(\mu,s,\beta,\varphi)$. Then
    \begin{equation}
    \sup_{\bP \in \cP}\bP(X \geq p)=
    \begin{cases}
     1, \quad &  p\in(0,\tau_1], \\
     w_p(p) + w_{\beta}(p), \quad & p\in[\tau_1,\tau_2], \\
     v_p(p), & p \in [\tau_2,\beta],
    \end{cases}
\end{equation}\label{eq:tailsup}
which is continuous in $p \in (0,\beta]$.
\end{proposition}
\begin{proof}
The problem of maximizing the supremum can be formulated as a semi-infinite  linear program:
\begin{align}\label{eq:sup_lp}
\begin{array}{ll}
    \sup_\P \; \; & \bE_{\P}(\1{\{x \geq p\}}) \\
    \text{s.t.} \; \; & \E_{\P}(\1_{[0,\beta]}(X)) = 1,\; \E_{\P}(X) = \mu,\; \E_{\P}(\varphi(X)) = s.
\end{array}
\end{align}
The dual of the supremum problem is given by
    \begin{equation}\label{dual1}
\begin{aligned}
&\inf_{\lambda_0,\lambda_1,\lambda_2 \in \R} &  &\lambda_0 + \lambda_1 \mu+\lambda_2 s\\
&\text{s.t.} &      & F(x) = \lambda_0  +\lambda_1 x+\lambda_2 \varphi(x)  \geq \1{\{x \geq p\}}, \ \forall x\in[0,\beta].
\end{aligned}
\end{equation}
If we can find feasible solutions for \eqref{eq:sup_lp} and \eqref{dual1} whose objective function values are equal to each other, then we may conclude from weak duality that both solutions are optimal for their respective problems, see, e.g., \cite{popescu2005semidefinite}.

    For $p \in (0,\tau_1]$ we have $\bP^*_2(X\geq p)=1$. Since $\bP^*_2(p)$ is feasible, it must be a solution. Next, consider $p \in [\tau_1,\tau_2]$. 
If we for now assume that $F(x) = \1{\{x \geq p\}}$ for the points $\{0,p,\beta\}$, this means that $\lambda_0, \lambda_1, \lambda_2$ satisfy the system
\begin{align}
    \left\{
\begin{array}{rrrrrrr}
     \lambda_0 & &  & + & \varphi(0) \lambda_2  &= & 0 \\
      \lambda_0 & + &\lambda_1 p  & + & \varphi(p) \lambda_2  &= & 1 \\
  \lambda_0 & + &\lambda_1 \beta  & + & \varphi(\beta) \lambda_2  &= & 1 \\
\end{array}\right..
\label{eq:equations_dual}
\end{align}
Solving this system gives
\begin{align}
\begin{array}{ll}
 \lambda_0 & = \displaystyle \frac{\varphi(0)(\beta-p)}{\beta(\varphi(0) - \varphi(p)) + p(\varphi(\beta) - \varphi(0))},\\[3.5ex] 
        \lambda_1 & = \displaystyle  \frac{\varphi(\beta) - \varphi(p)}{\beta(\varphi(0) - \varphi(p)) + p(\varphi(\beta) - \varphi(0))},  \\[3.5ex] 
        \lambda_2 & = \displaystyle  \frac{(p-\beta)}{\beta(\varphi(0) - \varphi(p)) + p(\varphi(\beta) - \varphi(0))}.
\end{array}
\label{eq:dual_variables} 
\end{align}
Note that these values are well-defined because the common denominator is strictly positive, since $\varphi$ is strictly convex. Furthermore, observe that $\lambda_2 < 0$, which means that $F(x)$ is concave. Then it follows that for all $x \in [0,\beta]$, we have $F(x) \geq \1{\{x \geq p\}}$. Furthermore, it can be checked that $\lambda_0 + \lambda_1 \mu + \lambda_2 s = w_p(p) + w_\beta(p)$. Then weak duality implies the result for $p \in [\tau_1,\tau_2]$ because of the (primal) feasibility of $\bP^*_3(p)$. Finally, consider $p \in [\tau_2,\beta]$. If we assume (and verify later) that $F(x) = \1{\{x \geq p\}}$ for the points $\{\alpha(p),p\}$ with $\alpha(p)$ as defined in \eqref{alpha}, and $F'(\alpha(p)) = 0$, we need $\lambda_0,\lambda_1,\lambda_2$ to satisfy 
\begin{align}
    \left\{
\begin{array}{rrrrrrr}
     \lambda_0 & & \lambda_1\alpha(p)  & + & \varphi(\alpha(p)) \lambda_2  &= & 0 \\
      \lambda_0 & + &\lambda_1 p  & + & \varphi(p) \lambda_2  &= & 1 \\
   & &\lambda_1  & + & \varphi'(\alpha(p)) \lambda_2  &= & 0 \\
\end{array}\right..
\label{eq:equations_dual2}
\end{align}
Solving this system gives
\begin{align}
\begin{array}{ll}
 \lambda_0 & = \displaystyle \frac{\alpha(p)\varphi'(\alpha(p))-\varphi(\alpha(p))}{\varphi(p)-(p-\alpha(p))\varphi'(\alpha(p))-\varphi(\alpha(p))},\\[3.5ex] 
        \lambda_1 & = \displaystyle  \frac{-\varphi'(\alpha(p))}{\varphi(p)-(p-\alpha(p))\varphi'(\alpha(p))-\varphi(\alpha(p))},  \\[3.5ex] 
        \lambda_2 & = \displaystyle  \frac{1}{\varphi(p)-(p-\alpha(p))\varphi'(\alpha(p))-\varphi(\alpha(p))}.
\end{array}
\label{eq:dual_variables} 
\end{align}
Since $\varphi(x)$ is strictly convex, it follows that $\lambda_2 > 0$. Hence, we know that $F(x)$ is a convex function, which means that for all $x \in [0,\beta]$ we have $F(x) \geq \1\{x\geq p\}$.
If we now check the dual objective, we obtain $$\lambda_0 + \lambda_1 \mu + \lambda_2 s = \frac{\alpha(p)\varphi'(\alpha(p))-\varphi(\alpha(p))-\varphi'(\alpha(p))\mu+s}{\varphi(p)-(p-\alpha(p))\varphi'(\alpha(p))-\varphi(\alpha(p))},$$ which does not reduce to $v_p(p)$. However, due to \eqref{alpha} we can substitute $s = \varphi(\alpha(p)) \frac{\mu-p}{\alpha(p)-p}+\varphi(p)\frac{\alpha(p)-\mu}{\alpha(p)-p}$, resulting in $$\lambda_0 + \lambda_1 \mu + \lambda_2 s = v_p(p).$$ Then weak duality implies the result for $p \in [\tau_2,\beta]$ because of the (primal) feasibility of $\bP^*_2(p)$. Next, notice that $w_{p}(\tau_1) + w_{\beta}(\tau_1)=1-w_0(\tau_1) = 1$ due to \eqref{eq:upsilon_1} and $w_p(\tau_2)+w_{\beta}(\tau_2)=1-w_0(\tau_2)=v_p(\tau_2)$ due to \eqref{eq:upsilon_2}. Hence, \eqref{eq:tailsup} is continuous in $p$. This completes the proof.
\end{proof}

Using a similar approach, we present analogous results for the other two quantities in \eqref{eq:fundq}, beginning with the worst-case tail-probability.
\begin{proposition}\label{prop:inf_tp}
    Consider ambiguity set $\cP = \cP(\mu,s,\beta,\varphi)$. Then
    \begin{equation}
    \inf_{\bP \in \cP}\bP(X \geq p)=
    \begin{cases}
     v_\alpha(p), \quad & p\in(0,\tau_1], \\
      w_{\beta}(p), \quad & p\in[\tau_1,\tau_2], \\
    0, & p \in[\tau_2,\beta],
    \end{cases}
\end{equation}\label{eq:tailinf}
which is continuous in $p \in (0,\beta]$.
\end{proposition}
The proof can be found in Appendix \ref{ap:proof_pdproof2}. Next, we present the result for the maximal conditional expectation, whose proof can be found in Appendix \ref{ap:proof_pdproof3}.

\begin{proposition}\label{prop:sup_ce}
    Consider ambiguity set $\cP = \cP(\mu,s,\beta,\varphi)$. Then
    \begin{equation}
    \sup_{\bP \in \cP}\bE(X | X \geq p)=
    \begin{cases}
     \alpha(p), \quad & p\in(0,\tau_1], \\
     \beta, & p \in[\tau_1,\beta],
    \end{cases}
\end{equation}\label{eq:supcond}
which is continuous in $p \in (0,\beta]$.
\end{proposition}

We remark that an interesting distinction arises when considering whether the dispersion constraint is imposed as an upper bound or as an exact value. When $\beta=\infty$, we can replace the actual dispersion value with an upper bound without loss of generality, as the worst-case distribution will always take on the upper bound value of the dispersion. However, when $\beta<\infty$, one has to be careful as this equivalence breaks down on the interval $p\in[\tau_1,\tau_2]$. Nonetheless, the analysis remains tractable. We provide a sketch of the proof, but leave the full derivation to the interested reader. We write $\bar{s}$ for the upper bound of the dispersion. The infimum of the tail-bound becomes \begin{equation*}
    \inf_{\bP \in \cP(\mu,\bar{s},\beta,\varphi)}\bP(X \geq p)=
    \begin{cases}
     v_\alpha(p), \quad & p\in(0,\tau_1], \\
      \frac{\mu-p}{\beta-p}, \quad & p\in[\tau_1,\mu], \\
    0, & p \in[\mu,\beta],
    \end{cases}
\end{equation*} by imposing the additional constraint $\lambda_2 \leq 0$ to dual program \ref{dual2}. This gives a different solution for $p \in [\tau_1,\tau_2]$. One can then select a two-point distribution on $\{p,\beta\}$ with corresponding dual solution $(\lambda_0,\lambda_1,\lambda_2) = (\frac{-p}{\beta-\mu},\frac{1}{\beta-p},0)$ when $p \in [\tau_1,\mu]$ to achieve objective value $\frac{\mu-p}{\beta-p}$, while the degenerate distribution on $\mu$ is feasible and clearly worst-case when $p \in [\mu,\tau_2]$. Without proof, we state that for $p\in (0,\mu]$, the quantity $\sup_{\bP\in\cP}\bP(X\geq p)=1$, while $\sup_{\bP\in\cP}\bE[X|X\geq p]$ remains unaltered. By using Theorem \ref{CRDT}, which can be verified to still hold under these conditions, we get
\begin{align*}
\inf_{\bP\in\cP(\mu,\bar{s},\beta,\varphi)}\textup{CR}(p,\bP) = \begin{cases}
    \inf_{\bP\in\cP(\mu,\bar{s},\varphi)}\textup{CR}(p,\bP), &\quad p \in (0,\tau_1]\\
    \inf_{\bP\in\cP(\mu,\beta)}\textup{CR}(p,\bP), &\quad p \in [\tau_1,\mu]\\
    0, &\quad p\in [\mu,\beta]
\end{cases}.
\end{align*} In what follows, we consider the dispersion $s$ as an exact value, consistent with the assumption made throughout the paper.

\subsection{Proof of Theorem \ref{CRDT}}\label{sec23pointsfundproof}
Let $p \in (\tau_2,\beta]$. Consider the two-point distribution $\bP_{\tau_2}$ supported on $0$ and $\tau_2$ and observe that $\textup{OPT}(\bP_{\tau_2})>0$ while $\textup{REV}(p,\bP_{\tau_2}) = 0.$ Hence, $$\inf_{\bP\in\cP}\textup{CR}(X\geq p) = 0.$$

Next, we will argue for $p \in (0,\tau_1]$ that $\bP^*_2(p^-)$ is the optimal solution and that when $p \in [\tau_1,\tau_2]$, then $\bP^*_3(p^-)$ is the optimal solution. For ease of writing we denote $$\bP^* = \bP^*(p^-) = 
\begin{cases}
\bP^*_2(p^-), & p \in (0, \tau_1], \\
\bP^*_3(p^-), & p \in [\tau_1, \tau_2],
\end{cases}
$$
and we mention that this can be rewritten in the following (intuitive) form:
$$\bP^*=
\left\{ \begin{array}{ll}
        0 & \text{w.p. } 1 - \sup_{\bP \in \cP} \bP(X \geq p^-), \\
        p^- & \text{w.p. } \sup_{\bP \in \cP} \bP(X \geq p^-) - \inf_{\bP \in \cP} \bP(X \geq p^-), \\
        y(p^-) & \text{w.p. } \inf_{\bP \in \cP} \bP(X \geq p^-),
    \end{array}\right.
$$
with $y(x) = \sup_{\bP \in \cP}\bE(X|X\geq x)$. Now, let $p \in (0,\tau_2]$ and consider $\bP^*\in\cP$. The competitive ratio will be evaluated in $\bP^*$ to serve as an upper bound. Observe that
\begin{align*}
    \textup{REV}(p,\bP^*) = p\bP^*(X\geq p) = p\bP^*(X = y(p^-))=p\inf_{\bP\in \cP}\bP(X\geq p^-)= p\inf_{\bP\in \cP}\bP(X\geq p),
\end{align*}
where the last equality follows from the continuity of \eqref{eq:tailinf}, and

\begin{align*}
\textup{OPT}(\bP^*) &= \sup_{t>0} t \bP^*(X \geq t) \\ 
&= \max\left\{p^-\bP^*(X \geq p^-), y(p^-)\bP^*(X \geq y(p^-))\right\} \\
&= \max\left\{p^-(\bP^*(X = p^-) + \bP^*(X = y(p^-))), y(p^-)\bP^*(X = y(p^-))\right\} \\
&= \max\left\{p^-\sup_{\bP\in \cP}\bP(X\geq p^-), y(p^-)\inf_{\bP\in \cP}\bP(X\geq p^-)\right\} \\
&= \max\left\{p\sup_{\bP\in \cP}\bP(X\geq p), y(p)\inf_{\bP\in \cP}\bP(X\geq p)\right\}.
\end{align*}
Hence,
\begin{align*}
    \textup{CR}(p,\bP^*) =
    \frac{\textup{REV}(p,\bP^*)}{\textup{OPT}(\bP^*)} = \min\left\{\frac{\inf_{\bP\in \cP}\bP(X\geq p)}{\sup_{\bP\in \cP}\bP(X\geq p)}, \frac{p}{y(p)}\right\}.
\end{align*}
 {Next, consider an arbitrary distribution from the ambiguity set $\bP_a \in \cP$ and verify that} $\textup{CR}(p,\bP_a) \geq \min\left\{\frac{\inf_{\bP\in \cP}\bP(X\geq p)}{\sup_{\bP\in \cP}\bP(X\geq p)}, \frac{p}{y(p)}\right\}$. Let $p^*$ denote an optimal take-it-or-leave-it price for $\bP_a$, so that $\textup{OPT}(\bP_a) = p^* \bP_a(X\geq p^*)$. 
 We next use a natural case distinction between $p^*\leq p$ and $p^*> p$, which was also used in \cite{chen2022distribution} and \cite{giannakopoulos2019robust}.
 First assume $p^* \leq p$. For this case, we need the following monotonicity result for the best-case revenue function, stated as a self-contained lemma, where we slightly abuse notation by allowing $p$ to vary:
 \begin{lemma}\label{lemma:g(t)}
    Let $p \in (0,\tau_2]$. The function 
    \begin{align}\label{function_g}
        g(p) = p \cdot \sup_{\Prob \in \cP}  \Prob(X \geq p)
    \end{align}
    is non-decreasing in $p$.
\end{lemma}

\noindent The best-case revenue function  \eqref{function_g} does not appear in 
\cite{chen2022distribution} and \cite{giannakopoulos2019robust}, and turns out to be essential for dealing with the additional challenges when incorporating the maximum valuation bound $\beta$ in the ambiguity set. 
The proof of Lemma~\ref{lemma:g(t)} is presented in Appendix \ref{app:a}. Employing Lemma~\ref{lemma:g(t)} yields the following lower bound:
\begin{align*}
    \textup{CR}(p;\bP_a) &\geq
    \frac{p\inf_{\bP \in \cP} \bP(X\geq p)}{\sup_{t \in (0,p]}g(t)} = \frac{\inf_{\bP \in \cP} \bP(X\geq p)}{\sup_{\bP \in \cP} \bP(X\geq p)}\\
    &\geq \min\left\{\frac{\inf_{\bP\in \cP}\bP(X\geq p)}{\sup_{\bP\in \cP}\bP(X\geq p)}, \frac{p}{y(p)}\right\}.
\end{align*}
 Next assume $p^* > p$ and observe that
\begin{align*}
\frac{p^* \bP_a(X \geq p^*)}{\bP_a(X \geq p)} &= \int_{\{x|x \geq p^*\}} \frac{p^*}{\bP_a(X \geq p)}\ {\rm d} \bP_a(x)\\
&\leq \int_{\{x|x \geq p^*\}} \frac{x}{\bP_a(X \geq p)}\ {\rm d} \bP_a(x)\\
&\leq \int_{\{x|x \geq p^*\}} \frac{x}{\bP_a(X \geq p)}\ {\rm d} \bP_a(x) + \int_{\{x|p \leq x < p^*\}} \frac{x}{\bP_a(X \geq p)}\ {\rm d} \bP_a(x)\\
&= \int_{\{x|x \geq p\}} \frac{x} {\bP_a(X \geq p)}\ {\rm d} \bP_a(x)\\
&= \bE_{\bP_a}(X|X \geq p).
\end{align*}\
\noindent Hence, the following tight lower bound can be constructed:
\begin{align*}
    \textup{CR}(p;\bP_a) &= \frac{p\bP_a(X \geq p)}{p^* \bP_a(X \geq p^*)} \geq
    \frac{p}{\bE_{\bP_a}[X|X \geq p]} \geq \frac{p}{y(p)}\\
    &\geq \min\left\{\frac{\inf_{\bP\in \cP}\bP(X\geq p)}{\sup_{\bP\in \cP}\bP(X\geq p)}, \frac{p}{y(p)}\right\}.
\end{align*}
The equality between the lower and upper bound proves the assertion.

\subsection{Proof of Theorem \ref{thm22}}
Let $p\in (0,\tau_2]$ and observe that $$\bP^*(X\geq p) = \bP^*(X=\sup_{\bP\in\cP}\bE(X|X\geq p^-)) = \inf_{\bP\in\cP}\bP(X\geq p^-) = \inf_{\bP\in\cP}\bP(X\geq p),$$ where the last equality follows from the continuity of \eqref{eq:tailinf}. Hence, $\bP^*(p^-)$ is a limiting solution of $\inf_{\bP\in\cP}\bP(X \geq p)$ and therefore a limiting solution of $\inf_{\bP\in\cP}\textup{REV}(p,\bP)$. Additionally, from Theorem \ref{CRDT} we know that $\textup{CR}(p,\bP^*) = \inf_{\bP\in\cP}\textup{CR}(p,\bP)$.
Let $p \in (\tau_2,\beta]$. From Theorem \ref{CRDT} we know that $\bP_{\tau_2} \in \arg\inf_{\bP\in\cP}\textup{CR}(p,\bP)$. Furthermore, since $p > \tau_2$, we have $$\bP_{\tau_2}(X\geq p) = 0 = \inf_{\bP\in\cP}\textup{REV}(p,\bP).$$ Therefore, $\bP_{\tau_2} \in \arg\inf_{\bP\in\cP}\textup{REV}(p,\bP)$. This finishes the proof.

The reason that both expected revenue and competitive ratio share the same solution for $p\in (0,\tau_2]$ can intuitively be understood by the fact that $\bP^*(p^-)$ solves both $$\inf_{\bP\in\cP}\bP(X\geq p) \; \text{ and } \; \sup_{\bP\in \cP}\bE(X|X \geq p),$$ while $\bP^*(p)$ solves $\sup_{\bP\in\cP}\bP(X\geq p)$, which exclusively plays a role in the $\textup{OPT}(\bP)$-part of the competitive ratio. Optimal expected revenue is unrestricted by the price of the seller, and can instead select $p^-$.

\section{Optimal robust pricing}\label{sec:rprice}
We now address the maximin ratio problem \eqref{CPR}, starting from the expression for the tight lower bound on the competitive ratio provided in Theorem~\ref{CRDT}. 
We first consider variance as dispersion measure in Section~\ref{sec:var} and then turn to fractional moments as dispersion measure in Section~\ref{sec:frac}. In both cases, we derive precise bounds for the three key quantities involved in the fundamental relation \eqref{decomposition}. These bounds allow us to obtain an exact expression for the worst-case competitive ratio. Next, we solve for the price that maximizes this worst-case CR. The optimal price is shown to be one of two possible candidates: a relatively low price and a relatively high price. We demonstrate that the optimal choice between these two prices is determined by the level of dispersion.

\subsection{Variance}\label{sec:var}
Consider the ambiguity set $\cP(\mu,\sigma,\beta)$ containing all distributions with mean $\mu$, variance $\sigma^2$ and maximal value $\beta$. This ambiguity set is a special case of \eqref{eq:ambiguityk} with $\varphi(x) = x^2$ and $s = \sigma^2 + \mu^2$. From $\eqref{eq:upsilon_1}$ and \eqref{eq:upsilon_2} we obtain $\tau_1 = \mu - \frac{\sigma^2}{\beta-\mu}$ and $\tau_2 = \mu + \frac{\sigma^2}{\mu}$, and the quantities $$\sup_{\bP\in\cP(\mu,\sigma,\beta)}\bP(X\geq p), \ \ \inf_{\bP\in\cP(\mu,\sigma,\beta)}\bP(X\geq p), \ \ \sup_{\bP\in\cP(\mu,\sigma,\beta)}\bE(X|X\geq p),$$ follow from Propositions \ref{prop:sup_tp}, \ref{prop:inf_tp} and \ref{prop:sup_ce}.
Substituting these quantities in Theorem \ref{CRDT} yields the tight lower bound for the competitive ratio
$$R(p,\sigma) = \inf_{\bP \in \cP(\mu,\sigma,\beta)}\textup{CR}(p,\bP) = \begin{cases}
     \min\left\{\frac{(\mu-p)^2}{(\mu-p)^2+\sigma^2},\frac{p(\mu-p)}{\mu(\mu-p)+\sigma^2}\right\}, \quad &  p\in(0,\tau_1], \\
     \min\left\{\frac{p(\mu^2+\sigma^2-p\mu)}{(\beta-p)(\mu(\beta+p-\mu)-\sigma^2)},\frac{p}{\beta}\right\}, \quad & p\in[\tau_1,\tau_2], \\
    0, \quad & p \in[\tau_2,\beta].
    \end{cases}$$ 
With this explicit expression for the worst-case ratio we can solve \eqref{CPR} and obtain the following characterization of the optimal prices (the proof is presented in Appendix \ref{appendix_variance}):
\begin{theorem}\label{th:optimal_price_variance}
    The optimal robust price $p_{(\mu,\sigma,\beta)}^*=\arg\sup_{p>0}\inf_{\bP\in\cP(\mu,\sigma,\beta)}\textup{CR}(p,\bP)\nonumber$ is
    \begin{align}
        p_{(\mu,\sigma,\beta)}^* = \begin{cases}
            p_l^*, \quad & \sigma \leq \sigma^{*}, \\
            p_h^* = \max\{p_{h1}^*,p_{h2}^*\}, \quad &    \sigma \geq \sigma^{*},
        \end{cases}
    \end{align}
with
\begin{align}
    p_{l}^*&=\mu - \sigma \cdot \left(\left(\frac{\mu}{2\sigma}+\sqrt{\frac{8}{27}+\frac{\mu}{2\sigma}}\right)^{\frac{1}{3}}+\left(\frac{\mu}{2\sigma}-\sqrt{\frac{8}{27}+\frac{\mu}{2\sigma}}\right)^{\frac{1}{3}}\right),\\
    p_{h1}^* &= \frac{1}{2}\left(\beta+\tau_2-\sqrt{(3\beta-\tau_2)^2-4\beta^2}\right),\ \ p_{h2}^* = \frac{\tau_2}{2}.
\end{align}
Furthermore, $\sigma^*$ is implicitly defined as the solution to
$R(p_l^*,\sigma^*) = R(p_h^*,\sigma^*)$. 
\end{theorem}

Theorem~\ref{th:optimal_price_variance} uncovers several new insights into robust pricing, which merit further reflection. First, we explain why and how the dispersion threshold, denoted by $\sigma^*$, plays a critical role. Figure~\ref{fig:pricevsratio} shows several worst-case CR plots as a function of price. These three curves correspond to the same mean and maximum valuation ($\mu=0.5$, $\beta=1$) but differ in their dispersion levels. In this particular case, $\sigma^*$ from Theorem~\ref{th:optimal_price_variance} numerically evaluates to $0.3194$, indicating that the seller should choose either a low or high price depending on whether the value of $\sigma$ is below or above $\sigma^*$. The three plots in Figure~\ref{fig:pricevsratio} clearly illustrate the root cause of this shift: the worst-case CR function has a global maximum that can occur at one of three locations, corresponding to the three prices in Theorem~\ref{th:optimal_price_variance}. The lowest price achieves the global maximum when $\sigma$ is sufficiently small, as is the case for $\sigma=0.3$. The $\sigma$ values in Figure~\ref{fig:pricevsratio} are deliberately chosen to be close to $\sigma^*$ to highlight this transition. For instance, when $\sigma=0.35$, the high price becomes optimal, and this effect is even more pronounced for $\sigma=0.4$. Figure~\ref{fig:pricevsratio} also aids in understanding why the transition from low to high pricing is discontinuous: the global maximum of the objective function shifts abruptly from one local maximum to another. This can be further understood by examining the worst-case CR, which has two distinct regimes for the optimal price: either $p \in (0,\tau_1]$, or $[\tau_1,\tau_2]$. The dispersion determines which regime is optimal. Although the worst-case CR remains continuous as the regime shifts, the optimal price does not, as it jumps abruptly from $p^*_l$ to $p^*_h$.

\begin{figure}[h!]
    \centering
    \includegraphics[width=0.6\linewidth]{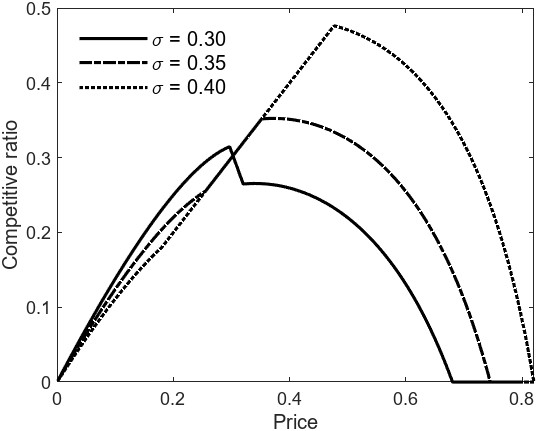}
     \caption{Function $p \mapsto R(p,\sigma)$ with $\mu = 0.5$ and $\beta = 1$.}
    \label{fig:pricevsratio}
\end{figure}

Let us further explore the main implications of Theorem~\ref{th:optimal_price_variance} and connect them to various established results in the pricing literature (Lemma  \ref{lem:rev_price})  . The  price denoted as $p^*_l$ decreases with $\sigma$, while $p^*_h$ increases as $\sigma$ grows. This behavior is illustrated in~Figure \ref{fig:2prices}. The sharp transition from the lower price $p^*_l$ to the higher price $p^*_h$ can be explained as follows: Low variance reflects a stable and relatively homogeneous market, where the optimal strategy for the seller is to set a lower price, prioritizing a higher conversion rate. Conversely, high variance indicates a more volatile market with less predictable demand, leading the seller to adopt a higher price, focusing on a smaller, more selective segment of the market. The threshold $\sigma^*$ represents the critical level of market risk at which the seller shifts from a mass-market approach to a niche market strategy. Although this shift can be intuitively understood, the specific prices and threshold $\sigma^*$ are influenced by subtle interactions among all the model parameters.

One part of the optimal price in Theorem~\ref{th:optimal_price_variance} has been discovered recently in two papers, using independent approaches. To explain this, consider 
$\cP(\mu,\sigma)$, the set of all distribution with mean $\mu$, variance $\sigma^2$ and no constraint on the maximal value. For this ambiguity set, the optimal price should follow from Theorem~\ref{th:optimal_price_variance} by letting $\beta\to\infty$. When we do so, we indeed recover a result recently obtained in 
\cite{giannakopoulos2019robust} and \cite{chen2022distribution}: 
\begin{align}
p_{(\mu,\sigma)}^* &= \arg\sup_{p>0}\inf_{\bP\in\cP(\mu,\sigma)}\textup{CR}(p,\bP)=p^*_l.
\end{align}
This shows that when valuations are unconstrained, the risk threshold $\sigma^*$ goes to infinity, so the high price  $p^*_h$ disappears and the low price $p^*_l$ becomes the unique price function. As a result, the competitive ratio will deteriorate towards zero when variance becomes maximal, rather than converge to one. This highlights the importance of placing a cap on valuations. When valuations can grow arbitrarily large, Nature will exploit this. Consequently, the max-min seller is forced to adopt an low pricing strategy, leading to a competitive ratio close to zero. However, if the seller knows in advance that valuations are capped at a certain maximum, a more favorable high-pricing strategy becomes viable. This strategy proves particularly relevant in situations with moderate to high valuation dispersion and leads to significant revenue improvements. Table~\ref{tabbb1} displays some numerical results for increasing $\sigma$-values, confirming the switch from low to high pricing. Observe that without the cap $\beta$, this results in a loss of performance when a high-valuation segment actually exists but is not accounted for. Table~\ref{tabbb2} shows results for a fixed variance and increasing maximal value, in which case high pricing proves optimal for sufficiently small $\beta$.

\begin{table}[ht!]
\centering
\begin{tabular}{@{}ccccc@{}}
\toprule
$\sigma$ & $p^*_{(\mu,\sigma)}$ & $p^*_{(\mu,\sigma,\beta)}$ & $R(p^*_{(\mu,\sigma)},\sigma)$ & $R(p^*_{(\mu,\sigma,\beta)},\sigma)$ \\ \midrule
0.00 & 0.5000 & 0.5000 & 1.0000 & 1.0000 \\
0.05 & 0.4076 & 0.4076 & 0.7734 & 0.7734 \\
0.10 & 0.3672 & 0.3672 & 0.6382 & 0.6382 \\
0.15 & 0.3404 & 0.3404 & 0.5310 & 0.5310 \\
0.20 & 0.3213 & 0.3213 & 0.4439 & 0.4439 \\
0.25 & 0.3073 & 0.3073 & 0.3728 & 0.3728 \\
0.30 & 0.2967 & 0.2967 & 0.3147 & 0.3147 \\
0.35 & 0.2886 & 0.3725 & 0.2886 & 0.3524 \\
0.40 & 0.2823 & 0.4763 & 0.2823 & 0.4763 \\
0.45 & 0.2773 & 0.6406 & 0.2773 & 0.6406 \\
0.50 & 0.2733 & 1.0000 & 0.2733 & 1.0000 \\ \bottomrule
\end{tabular}
\caption{Assessing the value of including $\beta$ in addition to $\mu$ and $\sigma$ when $\mu = 0.5$ and $\beta = 1$.}\label{tabbb1}
\end{table}

\begin{table}[]
\centering
\begin{tabular}{@{}ccccccccccc@{}}
\toprule
\multicolumn{1}{c}{\multirow{2}{*}{$\beta$}} & \multicolumn{1}{c}{} & \multicolumn{4}{c}{$\mu = 0.5$} & \multicolumn{1}{c}{} & \multicolumn{4}{c}{$\mu = 1$} \\ \cmidrule(lr){3-6} \cmidrule(l){8-11} 
\multicolumn{1}{c}{} & \multicolumn{1}{c}{} & \multicolumn{2}{c}{$p^*_{(\mu,\sigma,\beta)}$} & \multicolumn{1}{c}{} & \multicolumn{1}{c}{$R(p^*_{(\mu,\sigma,\beta)},\sigma)$} & \multicolumn{1}{c}{} & \multicolumn{2}{c}{$p^*_{(\mu,\sigma,\beta)}$} & \multicolumn{1}{c}{} & \multicolumn{1}{c}{$R(p^*_{(\mu,\sigma,\beta)},\sigma)$} \\ \midrule
1.0 &  & $p^*_{h1}$ & 1.0000 &  & 1.0000 &  & - & - &  & - \\
1.1 &  & $p^*_{h1}$ & 0.7146 &  & 0.6496 &  & - & - &  & - \\
1.2 &  & $p^*_{h1}$ & 0.6000 &  & 0.5000 &  & - & - &  & - \\
1.3 &  & $p^*_{h1}$ & 0.5077 &  & 0.3906 &  & $p^*_{h1}$ & 1.0188 &  & 0.7837 \\
1.4 &  & $p^*_{h2}$ & 0.5000 &  & 0.3086 &  & $p^*_{h1}$ & 0.8606 &  & 0.6147 \\
1.5 &  & $p^*_{h2}$ & 0.5000 &  & 0.2500 &  & $p^*_{h1}$ & 0.7500 &  & 0.5000 \\
1.6 &  & $p^*_{h2}$ & 0.5000 &  & 0.2066 &  & $p^*_{h1}$ & 0.6565 &  & 0.4103 \\
1.8 &  & $p^*_l$ & 0.2733 &  & 0.1705 &  & $p^*_l$ & 0.6145 &  & 0.3728 \\
2.0 &  & $p^*_l$ & 0.2733 &  & 0.1705 &  & $p^*_l$ & 0.6145 &  & 0.3728 \\
$\infty$ &  & $p^*_l$ & 0.2733 &  & 0.1705 &  & $p^*_l$ & 0.6145 &  & 0.3728 \\ \bottomrule
\end{tabular}
\caption{Impact of $\beta$ on $p^*_{(\mu,\sigma,\beta)}$ and $R(p^*_{(\mu,\sigma,\beta)},\sigma)$ when $\sigma = 0.5$.}\label{tabbb2}
\end{table}

We will next compare the maximin ratio price with the maximin revenue price in Theorem \ref{th:2prices}. The maximin revenue price was solved for the mean-variance ambiguity set $\cP(\mu,\sigma)$ by \cite{azar2012optimal} and more recently in \cite{suzdaltsev2020distributionally} and \cite{van2024robust} for $\cP(\mu,\sigma,\beta)$. Lemma \ref{lem:rev_price} precisely captures this result.
\begin{lemma}[\cite{azar2012optimal,suzdaltsev2020distributionally,van2024robust}]\label{lem:rev_price}
\begin{align}
\pi^*_{(\mu,\sigma,\beta)}&=\arg\sup_{p>0}\inf_{\bP\in\cP(\mu,\sigma,\beta)}\textup{REV}(p,\bP) \nonumber \\ &=\begin{cases}
            \pi^*_l = \mu - \sigma \cdot \large(\large(\tfrac{\mu}{\sigma}+\sqrt{1+(\tfrac{\mu}{\sigma})^2}\large)^{\frac{1}{3}} + \large(\tfrac{\mu}{\sigma}-\sqrt{1+(\tfrac{\mu}{\sigma})^2}\large)^{\frac{1}{3}}\large), \quad &  \sigma \leq \delta^{*} \\
            \pi^*_h = \beta - \sqrt{\beta(\beta-\mu-\frac{\sigma^2}{\mu})}, \quad & \sigma \geq \delta^{*}
        \end{cases}
\end{align}
with $\delta^* \in (0,\sigma_{\textup{max}})$ an implicit threshold value. 
\end{lemma}

Using these explicit pricing formulas, let us compare the max-min CR prices with their counterparts for expected revenue. \cite{chen2022distribution} makes this comparison for the case $\beta = \infty$, concluding that expected revenue gives rise to lower prices than the competitive ratio (i.e.,~$\pi^*_l < p^*_l$). However, the case $\beta < \infty$ allows for high-price strategies that were previously ruled out. Our next result demonstrates that this ordering is no longer true in general when $\beta < \infty$.
\begin{theorem}\label{th:2prices}
The optimal robust low prices are ordered as
    $\pi^*_l < p^*_l$, when $\sigma \leq \min\{\delta^*,\sigma^*\}$ and the optimal robust high prices are ordered as $\pi^*_h > p^*_h$, when $\sigma \geq \max\{\delta^*,\sigma^*\}$.
\end{theorem}
\begin{proof}
    One can observe that $p^*_l = p^*_{(\mu,\sigma)}$ when $\sigma < \sigma^*$ and $\pi^*_l = \pi^*_{(\mu,\sigma)}$ when $\sigma < \delta^*$. Hence, from \cite{chen2022distribution} we conclude that $\pi^*_l < p^*_l$. Next, we show that $\pi_h^* > p^*_h$. 
    Case 1: Assume $p^*_h = p^*_{h1}$. Then the ordering is true if, upon some rewriting $\tau_2 - \sqrt{(3\beta-\tau_2)^2-4\beta^2} < \beta - \sqrt{4\beta^2-4\beta\tau_2}$. Now clearly $\tau_2 < \beta$, so showing that $(3\beta-\tau_2)^2-4\beta^2 \geq 4\beta^2-4\beta\tau_2$ would suffice. This is equivalent to $(\beta-\tau_2)^2 \geq 0$, which is always true.
    Case 2: Assume $p^*_h = p^*_{h2}$. Then the ordering is true, because $\frac{1}{2}\tau_2 < \beta - \sqrt{\beta(\beta-\tau_2)} \iff \frac{1}{2}\tau_2^2 > 0$, which is always true. This proves the assertion.
\end{proof}

Figure~\ref{fig:2prices} illustrates how optimal prices differ under the expected revenue and competitive ratio objectives and, specifically, how the competitive ratio gives rise to more moderate pricing strategies. While both objectives involve the same underlying trade-off between price and conversion, they differ in how they prioritize these competing factors. This difference does not stem from Nature’s behavior, as by Theorem~\ref{thm22}, both objectives induce the same worst-case scenario for any fixed price. Instead, it stems from the structural differences between the two objectives. Applying Theorem~\ref{CRDT}, the worst-case competitive ratio for a fixed price can be expressed as
\begin{align*}
    \inf_{\bP\in\cP}\textup{CR}(p,\bP) =  \frac{p\inf_{\bP\in\cP}\bP(X\geq p)}{\max\{p\sup_{\bP\in\cP}\bP(X\geq p),y(p)\inf\bP(X\geq p)\}},
\end{align*}
where $y(p) = \sup_{\bP\in\cP}\bE[X|X\geq p] \geq \max\{\mu,p\}$. This expression reveals the limitations of extreme pricing under the competitive ratio objective. When the seller sets a low price, thereby achieving a high conversion rate, the optimal expected revenue achieves the same high conversion rate without having to sacrifice price. Conversely, if the seller sets a high price and accepts a low conversion rate, the optimal expected revenue sets the same price while achieving a superior conversion rate that remains bounded away from zero, even as the seller's conversion rate vanishes when $p\approx \tau_2$. Hence, while the expected revenue objective allows the seller to set more extreme prices depending on which factor of the trade-off dominates, the competitive ratio penalizes such extremities and instead encourages more conservative choices, leading to more balanced pricing strategies.

\begin{figure}[h!]
    \centering
    \includegraphics[width=0.6\linewidth]{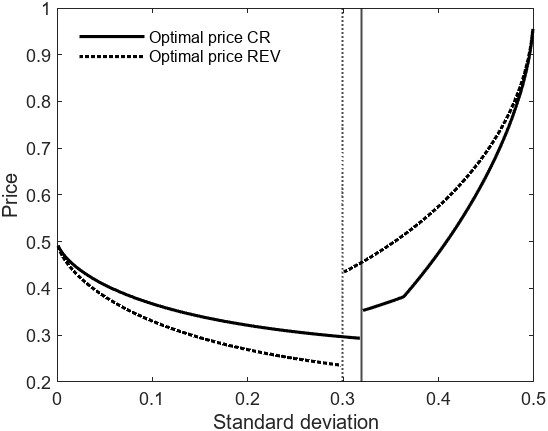}
    \caption{Proposition \ref{th:2prices} with $\mu = 0.5$ and $\beta = 1$. The vertical lines illustrate the location of $\delta^*$ and $\sigma^*$, respectively, from left to right.}
    \label{fig:2prices}
\end{figure}

\subsection{Fractional moments}\label{sec:frac}
Let $\cP(\mu,q,\beta)$ denote the specific instance of \eqref{eq:ambiguityk}, where $\varphi(x) = x^q$ with $q > 1$, and consider 
\begin{align}\label{pr_q}
\sup_{p>0}\inf_{\bP\in\cP(\mu,q,\beta)}\textup{CR}(p,\bP).
\end{align}
This setting contains the variance-setting as special case. We solve the inner minimization problem for fixed $p$ by leveraging Theorem \ref{CRDT} and then plugging in the quantities $$\sup_{\bP\in\cP(\mu,q,\beta)}\bP(X\geq p), \ \ \inf_{\bP\in\cP(\mu,q,\beta)}\bP(X\geq p), \ \ \sup_{\bP\in\cP(\mu,q,\beta)}\bE(X|X\geq p)$$ from Proposition \ref{prop:sup_tp}, \ref{prop:inf_tp}, and \ref{prop:sup_ce}. This results in
\begin{equation}\label{inner_solution}
    \inf_{\bP\in \cP(\mu,q,\beta)} \textup{CR}(p,\bP) =\begin{cases}
     \min\left\{\frac{\mu-p}{\alpha(p)-p},\frac{p}{\alpha(p)}\right\}, \quad & p\in(0,\tau_1], \\
     \min\left\{\frac{ps-\mu p^q}{\mu(\beta^q-p^q)-s(\beta-p)},\frac{p}{\beta}\right\}, \quad &  p\in[\tau_1,\tau_2], \\
    0, & p \in[\tau_2,\beta],
    \end{cases}
\end{equation}
where $\alpha(p)$ is defined as in \eqref{alpha}. The value $\tau_1$ results from $\eqref{eq:upsilon_1}$, although this does not lead to a closed-form expression, while \eqref{eq:upsilon_2} gives $\tau_2 = (\frac{s}{\mu})^{\frac{1}{q-1}}$. 
To facilitate analysis, we define the functions
\begin{equation}\label{functions_g}
    g_{1a}(p) = \frac{\mu-p}{\alpha(p)-p},\ \ g_{1b}(p) = \frac{p}{\alpha(p)},\ \ g_{2a}(p) = \frac{ps-\mu p^q}{\mu(\beta^q-p^q)-s(\beta-p)}, \ \ g_{2b}(p) = \frac{p}{\beta}.
\end{equation}
In order to solve \eqref{pr_q}, we need to understand the shapes of \eqref{functions_g}.
One can see that $g_{1a}(p)$ is decreasing in $p$, since $\alpha(p) \geq \mu$ is increasing in $p$ for $p\leq \tau_1<\mu$. Furthermore, $g_{2b}(p)$ is clearly increasing in $p$. Additionally, although not immediately clear, both $g_{1b}(p)$ and $g_{2a}(p)$ have a unique local maximum location on their respective intervals. The next lemma establishes this:
\begin{lemma}\label{lemma:ulm}
    The functions $g_{1b}(p)$ with $p \in (0,\tau_1]$ and $g_{2a}(p)$ with $p \in [\tau_1,\tau_2]$ have one local maximum location.
\end{lemma}
\noindent The proof is presented in Appendix \ref{appendix_powermoments}. Now that we have derived the shape of \eqref{inner_solution}, we can solve the outer maximization problem. The next theorem characterizes the resulting optimal solution:
\begin{theorem}\label{th:optimal_price_q}
    The optimal robust price is one of four candidate prices: $$p_{(\mu,q,\beta)}^*=\arg\sup_{p>0}\inf_{\bP\in\cP(\mu,q,\beta)}\textup{CR}(p,\bP)\in \{\min\{\bar{p}_l,\hat{p}_l\},\max\{\bar{p}_h,\hat{p}_h\}\}$$
with
\begin{enumerate}[label=\textup{(\roman*)}]
\item $\bar{p}_l$: Left-most solution in $(0,\mu]$ to $\bar{p}_l = \alpha(\bar{p}_l) - \sqrt{\alpha(\bar{p}_l)(\alpha(\bar{p}_l) - \mu)}$,
\item $\hat{p}_l$: Solution to $\frac{\alpha(\hat{p}_l)^q - \hat{p}_l^q}{\alpha(\hat{p}_l) - \hat{p}_l} = \frac{qs}{\mu}$,
\item $\bar{p}_h$: Right-most solution in $(0,\tau_2]$ to $\frac{\beta^q - \bar{p}_l^q + \beta \bar{p}_l^{q-1}}{2\beta - \bar{p}_l} = \frac{s}{\mu}$,
\item $\hat{p}_h$: Defined as $\hat{p}_h = \left(\frac{s}{q\mu}\right)^{\frac{1}{q-1}}$.
\end{enumerate}
\end{theorem}
\begin{proof}
    Consider $p \in (0,\tau_1]$. Since $g_{1a}(p)$ is decreasing in $p$, $g_{1b}(p)$ has one maximum location, and $g_{1b}(0^+)=0 < g_{1a}(0^+)$, we know that either the left-most intersection point between $g_{1a}(p)$ and $g_{1b}(p)$, which is $\bar{p}_l$, the maximum location of $g_{1b}(p)$, which is $\hat{p}_l$, or the boundary solution $\tau_1$ is the optimal price. We now consider three cases: \text{Case 1:} $\min\{\bar{p}_l,\hat{p}_l,\tau_1\} = \bar{p}_l$. Then $\min\{g_{1a}(p),g_{1b}(p)\} = g_{1b}(p)$ when $p \in (0,\bar{p}_l]$ and $\min\{g_{1a}(p),g_{1b}(p)\}$ is decreasing in $p$ when $p \in [\bar{p}_l,\tau_1]$. Furthermore, $g_{1b}(p)$ is increasing in $p$ on the interval $(0,\bar{p}_l]$. Hence, the maximum location is attained at $\bar{p}_l$. \text{Case 2:} $\min\{\bar{p}_l,\hat{p}_l,\tau_1\} = \hat{p}_l$. Then $\min\{g_{1a}(p),g_{1b}(p)\} = g_{1b}(p)$ when $p \in (0,\bar{p}_l]$ and $\min\{g_{1a}(p),g_{1b}(p)\}$ is decreasing in $p$ when $p \in [\bar{p}_l,\tau_1]$. Hence, the maximum location is attained at $\hat{p}_l$. \text{Case 3:} $\min\{\hat{p}_l,\bar{p}_l,\tau_1\} = \tau_1$. Clearly, $\min\{g_{1a}(p),g_{1b}(p)\}=g_{1b}(p)$ with $g_{1b}(p)$ increasing in $p$. Hence, the maximum location is attained at $\tau_1$. By combining these three cases we conclude that the maximum location on the interval $(0,\tau_1]$ is attained at $\min\{\bar{p}_l,\hat{p}_l,\tau_1\}$. 
    
    Next, consider $p \in [\tau_1,\tau_2]$. We know that $g_{2a}(p)$ has one maximum location and that $g_{2b}(p)$ is increasing in $p$. Hence, we deduce that the optimal value is either attained by the right-most intersection point between $g_{2a}(p)$ and $g_{2b}(p)$, which is $\bar{p}_h$, the maximum location of $g_{2a}(p)$, which is $\hat{p}_h$, or the boundary solution $\tau_1$. We now again consider three cases: \text{Case 1:} $\max\{\bar{p}_h,\hat{p}_h,\tau_1\} = \bar{p}_h$. If $g_{2a}(\tau_1) \geq g_{2b}(\tau_1)$, then $\min\{g_{2a}(p),g_{2b}(p)\}$ is increasing in $p$ when $p \in [\tau_1,\bar{p}_h]$ and $\min\{g_{2a}(p),g_{2b}(p)\}$ is decreasing in $p$ when $p \in [\bar{p}_h,\tau_2]$. Hence, the maximum location is attained at $\bar{p}_h$. Similarly, if $g_{2a}(\tau_1) < g_{2b}(\tau_1)$, then $\min\{g_{2a}(p),g_{2b}(p)\}$ is increasing when $p \in [\tau_1,\bar{p}_h]$, and $\min\{g_{2a}(p),g_{2b}(p)\}$ is decreasing in $p$ when $[\bar{p}_h,\tau_2]$. Hence, the maximum location is attained at $\bar{p}_h$. \text{Case 2:} $\max\{\bar{p}_h,\hat{p}_h,\tau_1\} = \hat{p}_h$. If $g_{2a}(\tau_1) \geq g_{2b}(\tau_1)$, then $\min\{g_{2a}(p),g_{2b}(p)\}$ is increasing in $p$ when $p \in [\tau_1,\bar{p}_h]$ and $\min\{g_{2a}(p),g_{2b}(p)\}$ is decreasing in $p$ when $p \in [\bar{p}_h,\tau_2]$. Hence, the maximum location is attained at $\hat{p}_h$. Similarly, if $g_{2a}(\tau_1) < g_{2b}(\tau_1)$, then $\min\{g_{2a}(p),g_{2b}(p)\}$ is increasing in $p$ when $p \in [\tau_1,\tau_2]$. Hence, the maximum location is attained at $\hat{p}_h$. \text{Case 3:} $\max\{\bar{p}_h,\hat{p}_h,\tau_1\} = \tau_1$. Clearly, $g_{2a}(p)$ is now decreasing on the entire interval $[\tau_1,\tau_2]$. Hence, it must be that $g_{2a}(\tau_1) < g_{2b}(\tau_1)$, as otherwise the right-most intersection would occur after $\tau_1$. But then $\min\{g_{2a}(p),g_{2b}(p)\} = g_{2a}(p)$ and since this is decreasing in $p$ on $[\tau_1,\tau_2]$, the maximum location is attained at $\tau_1$. By combining these three cases we conclude that the maximum location on the interval $[\tau_1,\tau_2]$ is attained at $\max\{\bar{p}_h,\hat{p}_h,\tau_1\}$. 
    
    However, $\tau_1$ is never the global maximum location, unless $\tau_1 = \bar{p}_l = \bar{p}_h$. Assume $\tau_1$ is the global maximum location which differs from $\bar{p}_l$ and $\bar{p}_h$. Now notice that $\tau_1$ is part of both the intervals $(0,\tau_1]$ and $[\tau_1,\tau_2]$, meaning it can only be the global maximum location if it is the local maximum location in both intervals. Hence, it holds that $\bar{p}_h < \tau_1 < \bar{p}_l$ and $\hat{p}_h < \tau_1 < \hat{p}_l$. Now due to $\tau_1 < \hat{p}_l$, we know that $g_{1b}(p)$ is increasing in $p$ on $(0,\tau_1]$ and since $\hat{p}_h < \tau_1$, we know that $g_{2a}(p)$ is decreasing in $p$ on $[\tau_1,\tau_2]$. Furthermore, $g_{1a}(0^+) = \frac{\mu}{\tau_2}>0=g_{1b}(0)$, which implies that $g_{1a}(\tau_1) > g_{1b}(\tau_1)$. However, due to $g_{1a}(\tau_1) = g_{2a}(\tau_1)$ and $g_{1b}(\tau_1) = g_{2b}(\tau_1)$, it must be that $\bar{p}_h \geq \tau_1$. This is a contradiction.
\end{proof}

\begin{figure}[h!]
    \centering
    \begin{subfigure}[b]{0.48\linewidth}
        \centering
        \includegraphics[width=\linewidth]{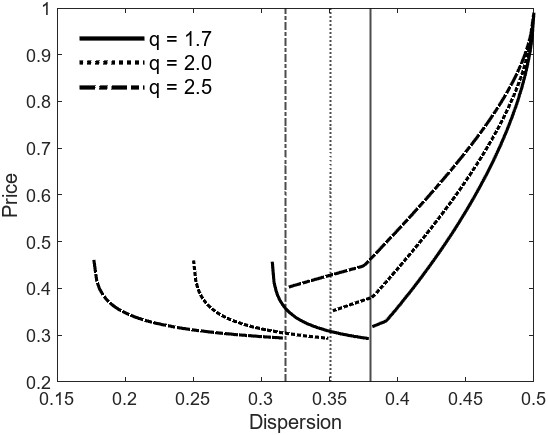}
        \caption{Optimal price as a function of dispersion. Vertical lines indicate discontinuous price jumps.}
    \end{subfigure}
    \hfill
    \begin{subfigure}[b]{0.48\linewidth}
        \centering
        \includegraphics[width=\linewidth]{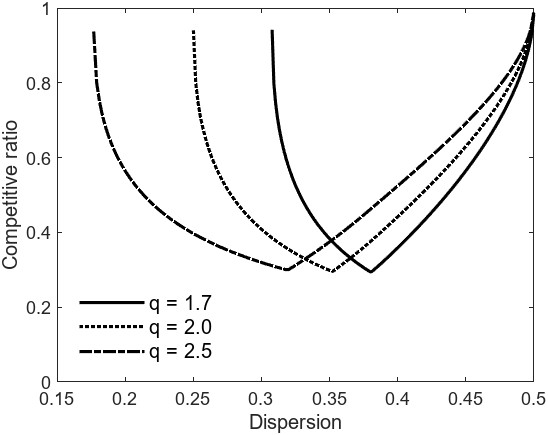}
        \caption{Optimal competitive ratio as a function of dispersion.}
        \label{fig:3prices_var}
    \end{subfigure}
    \caption{Optimal price and competitive ratio for various $q$ with $\mu = 0.5$ and $\beta = 1$, based on Theorem~\ref{th:optimal_price_q}.}
    \label{fig:q-moment}
\end{figure}

Figure~\ref{fig:q-moment} shows how dispersion knowledge impacts pricing. As $q$ increases, the pricing strategy extends over a larger range of dispersion values. The discontinuous jump in price also seems to become larger.

\section{Conclusions}\label{sec:conclusions}
We have solved the deterministic monopoly pricing problem with competitive ratio objective and limited information. Its fractional shape, with a numerator and denominator both responding to the same scenario, makes the critical ratio stand out compared with alternative criteria such as expected revenue or regret.
Our approach utilizes a max-min strategy, where the seller selects the optimal price while anticipating the worst-case scenario that complies with the available information.

In contrast to Bayesian pricing, which assumes a known prior for the valuation distribution, our max-min approach allows for ambiguity. We assume the prior belongs to a set of distributions characterized by the same mean, dispersion, and maximal value, though the set may contain infinitely many priors. These three characteristics---mean, dispersion, and maximal value---constitute both the seller’s knowledge and the constraints Nature must respect when determining the worst-case scenario. Identifying this worst-case scenario posed the primary mathematical challenge tackled in this paper.

We characterized the worst-case scenario by introducing a novel proof technique. Rather than solving a semi-infinite fractional linear program, we leveraged the intricate relationship between the competitive ratio and three fundamental quantities related to tail bounds and conditional expectations of the valuation distribution. By establishing tight bounds for these quantities, we derived a tight bound for the competitive ratio in Theorem~\ref{CRDT}. 
A key property is that for the ambiguity sets considered in this paper, tight bounds are typically attained by (mixtures of) two- and three-point distributions.
Another key concept, subtly embedded in the proof, is the best-case revenue function, described in Lemma~\ref{lemma:g(t)}. This function must be non-increasing up to a certain value for the proof of Theorem~\ref{CRDT} to work.

Solving the minimization problem yielded the worst-case competitive ratio, which was then used to determine the max-min optimal prices. This led to several new insights and guidelines. We illustrated how optimal prices respond to changes in dispersion, and particularly how knowledge of lower moments translates into more extreme pricing strategies.
We also found that introducing a cap on valuations, alongside mean and dispersion information, mitigates the risk of the competitive ratio approaching zero under high dispersion. Without such a cap, the seller is forced into a low-pricing strategy, as valuations can grow arbitrarily large. By contrast, knowing that valuations are limited by a maximum value enables the seller to explore more advantageous high-pricing strategies that prevent the competitive ratio from approaching zero.

\begin{small}
\bibliographystyle{apalike}
\bibliography{bibbook}
\end{small}

\newpage
\appendix
\section{Proof of Lemma \ref{lemma:non-empty}}\label{ap:proof_lne}
First assume that $\mu \leq \tau_2 \leq \beta$. Then by construction the two-point distribution supported on $\{0,\tau_2\}$ is contained in $\cPpar$, and, hence, this set is non-empty. Next, assume that $\cPpar$ is non-empty. Consider the semi-infinite linear program, see, e.g., \cite{popescu2005semidefinite}, 

\begin{align}\label{eq:three_lp}
\begin{array}{ll}
    \max \; \; & 0 \\
    \text{subject to} \; \; & \E_{\P}(\1_{[0,\beta]}(X)) = 1,\; \E_{\P}(X) = \mu,\; \E_{\P}(\varphi(X)) = s
\end{array}
\end{align}
Because of the non-emptiness of $\cPpar$, it is known that the above program is feasible and, hence, e.g., \cite[Theorem 1]{rogosinski1958moments} or \cite[Lemma 3.1]{shapiro2001duality} yields that there must exist an optimal three-point distribution for \eqref{eq:three_lp}. That is, there exists a three-point distribution $\P_{xyz} \in \cPpar$ supported on points $\{x,y,z\}$ with $0 \leq x \leq y \leq z \leq \beta$ and $x \leq \mu \leq z$. 

We will next argue that there must then also exist a two-point distribution in $\cPpar$. We can increase both $x$ and $y$ by shifting probability mass from $y$ to $x$ in the fashion of the two-point distribution characterization given earlier. Throughout this process, the point $z$ and its mass $v_z$ are kept fixed. To be precise, the probability masses $v_x$ and $v_y$ of the points $x$ and $y$ satisfy the system, 
\begin{align}
    \left\{
\begin{array}{rrrrl}
     v_x & + & v_y &= & 1 - v_z \\
      x v_x  & + & y v_y  &= & \mu - z v_z  \\
     \varphi(x) v_x & + & \varphi(y)v_y &  = & s - \varphi(z) v_z\\
     v_x, &  & v_y & \geq & 0 \\
\end{array}\right. .
\label{eq:equations_threepoint}
\end{align}
We emphasize that $v_z$ and $z$ are considered fixed. We can express $v_x, v_y$ and $y$ as functions of $x$, in a similar spirit as the characterization of two-point distributions given by \eqref{alpha} and \eqref{eq:two_point_masses}. If we let $x$ increase then $y$ increases as well, and the probability masses change in such a way that the overall distribution remains in $\cPpar$. We keep increasing $x$ until $y$ becomes equal to $z$, in which case we have arrived at a two-point distribution; let us call it $\P_{xz}$. {Since the two-point distribution on $\{0,\tau_2\}$ is the two-point distribution whose right support is smallest among all two-point distributions in $\cP(\mu,s,\beta,\varphi)$}, it must be the case that $\mu \leq \tau_2 \leq z$. But since $z$ has been kept fixed throughout this procedure, we also know that $\tau_2 \leq z \leq \beta$ still holds, and therefore $\tau_2 \leq \beta$.

\section{Proof of Lemma \ref{lem:nn3pd}}\label{ap:proof_nn3pd}
Due to strict convexity of $\varphi$, the common denominator $\beta(\varphi(0) - \varphi(p)) + p(\varphi(\beta) - \varphi(0))$ of the three masses is always strictly positive. Now consider the numerator of $w_0(p)$ and observe that $w_0(p) \geq 0$ if and only if
\begin{align*}
    s\geq \varphi(p)\cdot\frac{\beta-\mu}{\beta-p} + \varphi(\beta)\cdot \frac{\mu-p}{\beta-p} = f(p).
\end{align*}
This holds if $f(p)$ is decreasing in $p$ for $p \in [\tau_1,\tau_2]$, as $f(\tau_1) = s$. We will show this by proving that $f'(p) \leq 0$. Note that
$$f'(p) = \dfrac{\left({\mu}-{\beta}\right)\left(\left(p-{\beta}\right)\varphi'\left(p\right)-\varphi\left(p\right)+\varphi(\beta)\right)}{\left(p-{\beta}\right)^2},$$
and, because $(\mu - \beta) \leq 0$ and $(p-\beta)^2 \geq 0$, we have $f'(p) \leq 0$ if and only if $\left(p-{\beta}\right)\varphi'\left(p\right)-\varphi\left(p\right)+\varphi(\beta) \geq 0$. The latter is equivalent to
$$
\varphi'(p) \leq \frac{\varphi(\beta)-\varphi(p)}{\beta-p},
$$
which is true because of the (strict) convexity of $\varphi$. Next, consider the numerator of $w_p(p)$ and notice that $w_p(p)\geq 0$ if and only if
\begin{align*}
    \frac{s-\varphi(0)}{\mu} \leq \frac{\varphi(\beta)-\varphi(0)}{\beta}.
\end{align*} Again, this is the case, since $\tau_2 \leq \beta$ and convexity of $\varphi$ implies that $$\frac{s-\varphi(0)}{\mu} = \frac{\varphi(\tau_2)-\varphi(0)}{\tau_2} \leq \frac{\varphi(\beta)-\varphi(0)}{\beta}.$$ Finally, consider the numerator of $w_{\beta}(p)$ and observe that $w_{\beta}(p)\geq 0$ if and only if
\begin{align*}
    \frac{s-\varphi(0)}{\mu} \geq \frac{\varphi(p)-\varphi(0)}{p}.
\end{align*} 
Once more, this is the case, since $p \leq \tau_2$ and convexity of $\varphi$ implies that 
\begin{align*}
    \frac{\varphi(p)-\varphi(0)}{p} \leq \frac{\varphi(\tau_2)-\varphi(0)}{\tau_2} = \frac{s-\varphi(0)}{\mu}.
\end{align*}
This completes the proof.

\section{Proof of Proposition \ref{prop:inf_tp}}\label{ap:proof_pdproof2}
For convenience, we solve $\inf_{\bP \in \cP}\bP(X > p)$ and later conclude that this quantity equals $\inf_{\bP \in \cP}\bP(X \geq p)$. The dual of this infimum problem is given by
    \begin{equation}\label{dual2}
\begin{aligned}
&\sup_{\lambda_0,\lambda_1,\lambda_2 \in \R} &  &\lambda_0 + \lambda_1 \mu+\lambda_2 s,\\
&\text{s.t.} &      & F(x) = \lambda_0  +\lambda_1 x+\lambda_2 \varphi(x)  \leq \1{\{x > p\}}, \ \forall x\in[0,\beta]. 
\end{aligned}
\end{equation}
Let $p \in (0,\tau_1]$. If we assume (and verify later) that $F(x) = \1{\{x > p\}}$ for the points $\{p,\alpha(p)\}$ with $\alpha(p)$ as defined in \eqref{alpha}, and $F'(\alpha(p)) = 0$, we need $\lambda_0,\lambda_1,\lambda_2$ to satisfy 
\begin{align}
    \left\{
\begin{array}{rrrrrrr}
      \lambda_0 & + &\lambda_1 p  & + & \varphi(p) \lambda_2  &= & 0 \\
      \lambda_0 & & \lambda_1\alpha(p)  & + & \varphi(\alpha(p)) \lambda_2  &= & 1 \\
   & &\lambda_1  & + & \varphi'(\alpha(p)) \lambda_2  &= & 0 \\
\end{array}\right..
\label{eq:equations_dual2}
\end{align}
Solving this system gives
\begin{align}
\begin{array}{ll}
 \lambda_0 & = \displaystyle \frac{\varphi(p)-p\varphi'(\alpha(p))}{\varphi(p)-(p-\alpha(p))\varphi'(\alpha(p))-\varphi(\alpha(p))},\\[3.5ex] 
        \lambda_1 & = \displaystyle  \frac{\varphi'(\alpha(p))}{\varphi(p)-(p-\alpha(p))\varphi'(\alpha(p))-\varphi(\alpha(p))},  \\[3.5ex] 
        \lambda_2 & = \displaystyle  -\frac{1}{\varphi(p)-(p-\alpha(p))\varphi'(\alpha(p))-\varphi(\alpha(p))}.
\end{array}
\label{eq:dual_variables} 
\end{align}
Since $\varphi(x)$ is strictly convex, it follows that $\lambda_2 < 0$. Hence, $F(x)$ is a concave function, which means that also for all $x \in [0,\beta]$ we have $F(x) \leq \1\{x> p\}$.
If we now check the dual objective, we obtain $$\lambda_0 + \lambda_1 \mu + \lambda_2 s = \frac{\varphi(p)-p\varphi'(\alpha(p))+\varphi'(\alpha(p))\mu-s}{\varphi(p)-(p-\alpha(p))\varphi'(\alpha(p))-\varphi(\alpha(p))},$$ which does not reduce to $v_\alpha(p)$. However, due to \eqref{alpha} we can substitute $s = \varphi(\alpha(p)) \frac{\mu-p}{\alpha(p)-p}+\varphi(p)\frac{\alpha(p)-\mu}{\alpha(p)-p}$, resulting in $$\lambda_0 + \lambda_1 \mu + \lambda_2 s = v_\alpha(p).$$ Then weak duality implies the result for $p \in (0,\tau_1]$ because of the (primal) feasibility of $\bP^*_2(p)$.
Next, consider $p \in [\tau_1,\tau_2]$. 
If we for now assume that $F(x) = \1{\{x > p\}}$ for the points $\{0,p,\beta\}$, this means that $\lambda_0, \lambda_1, \lambda_2$ satisfy the system
\begin{align}
    \left\{
\begin{array}{rrrrrrr}
     \lambda_0 & &  & + & \varphi(0) \lambda_2  &= & 0 \\
      \lambda_0 & + &\lambda_1 p  & + & \varphi(p) \lambda_2  &= & 0 \\
  \lambda_0 & + &\lambda_1 \beta  & + & \varphi(\beta) \lambda_2  &= & 1 \\
\end{array}\right..
\label{eq:equations_dual}
\end{align}
Solving this system gives
\begin{align}
\begin{array}{ll}
 \lambda_0 & = \displaystyle \frac{-p\varphi(0)}{\beta(\varphi(0) - \varphi(p)) + p(\varphi(\beta) - \varphi(0))},\\[3.5ex] 
        \lambda_1 & = \displaystyle  \frac{\varphi(0)-\varphi(p)}{\beta(\varphi(0) - \varphi(p)) + p(\varphi(\beta) - \varphi(0))},  \\[3.5ex] 
        \lambda_2 & = \displaystyle  \frac{p}{\beta(\varphi(0) - \varphi(p)) + p(\varphi(\beta) - \varphi(0))}.
\end{array}
\label{eq:dual_variables} 
\end{align}
Since $\varphi(x)$ is strictly convex, it follows that $\lambda_2 > 0$. Hence, $F(x)$ is convex. Then it follows that for all $x \in [0,\beta]$, we have $F(x) \leq \1{\{x > p\}}$. Furthermore, it can be checked that $\lambda_0 + \lambda_1 \mu + \lambda_2 s = w_\beta(p)$. Then weak duality implies the result for $p \in [\tau_1,\tau_2]$ because of the (primal) feasibility of $\bP^*_3(p)$. For $p \in (\tau_2,\beta]$ consider the two-point distribution $\bP_{\tau_2}$ supported on $0$ and $\tau_2$. Since $\bP_{\tau_2}(X\geq p)=0$, it must be a solution.
A standard result, e.g., \cite[Lemma A1]{van2024robust}, ensures that if $\inf_{\bP \in \cP}\bP(X > p)$ is continuous in $p$, then $\inf_{\bP \in \cP}\bP(X > p) = \inf_{\bP \in \cP}\bP(X \geq p)$. Since $w_{\beta}(\tau_1) = v_{\alpha}(\tau_1)$ due to \eqref{eq:upsilon_1} and because $w_{\beta}(\tau_2) = 0$, we conclude that continuity holds.

\section{Proof of Proposition \ref{prop:sup_ce}}\label{ap:proof_pdproof3}
Let $p \in (0,\tau_1]$. Denote $\cP(\mu,s,\varphi) = \{\Prob :  \E_{\P}(1) = 1,\; \E_{\P}(X) = \mu, \E_{\P}(\varphi(X)) = s, \; 0 \leq X < \infty\}$. If there exists a feasible two-point distribution supported on $\{p,\alpha(p)\}$ with $\alpha(p) > p$ and $\alpha(p)$ the solution to \eqref{alpha}, then it follows from \cite[Proposition 3]{van2023generalized} and \cite[Theorem 3.1]{kleer2024distribution} that $$\sup_{\bP \in \cP(\mu,s,\varphi)}\bE(X|X\geq p) = \alpha(p).$$ Since $\bP^*_2(p^-)\in \cP$, we obtain
    $$\alpha(p) = \bE_{\bP^*_2}(X|X\geq p) \leq \sup_{\bP \in \cP}\bE(X|X\geq p)\leq \sup_{\bP \in\cP(\mu,s,\varphi)}\bE(X|X\geq p) = \alpha(p).$$
    Let $p \in (\tau_1,\beta]$. Now consider $\bP^*_3(p^-)$. Then $$\beta = \bE_{\bP^*_{3}}(X|X\geq p^-) \leq  \sup_{\bP \in \cP}\bE(X|X\geq p) \leq \beta.$$ Additionally, note that $\alpha(\tau_1)=\beta$ due to \eqref{eq:upsilon_1}. Hence, we conclude that continuity holds.
    
\section{Proof of Lemma \ref{lemma:g(t)}}\label{app:a}
From Proposition \ref{prop:sup_tp} we know that $g(p) = p$ when $p \in (0,\tau_1]$, which is clearly non-decreasing. We now consider the case when $p \in [\tau_1,\tau_2]$. From Proposition \ref{prop:sup_tp} we obtain
\begin{align}
 g(p) = p\cdot(w_p(p) + w_{\beta}(p)) = \frac{p\cdot \left[(\varphi(0) - s)(\beta-p) + \mu(\varphi(\beta) - \varphi(p)) \right]}{\beta(\varphi(0) - \varphi(p)) + p(\varphi(\beta) - \varphi(0))} = \frac{q(p)}{r(p)}.
\label{eq:sup_problem_1}
\end{align}
In order to show that \eqref{eq:sup_problem_1} is non-decreasing, we take the derivative of the expression on the right-hand side and show it is non-negative. That is, using the quotient rule, we have to show that
\begin{align}
\frac{r(p)q'(p) - q(p)r'(p)}{r(p)^2} \geq 0.
    \label{eq:quotient_rule}
\end{align}
This is the same as showing that $r(p)q'(p) - q(p)r'(p) \geq 0$, i.e, 
\begin{align}
[\beta(\varphi(0) - \varphi(p)) &+ p(\varphi(\beta) - \varphi(0))]\cdot[(\varphi(0)-s)(\beta-p) + \mu(\varphi(\beta) - \varphi(p))- p(\varphi(0)-s+ \mu\cdot \varphi'(p))] \nonumber  \\
&- [(\varphi(\beta) - \varphi(0)) - \beta \cdot \varphi'(p)][(\varphi(0)-s)p(\beta-p) + \mu \cdot p(\varphi(\beta) - \varphi(p))] \geq 0.
\end{align}

\noindent Writing out this expression results in some terms canceling out against each other, after which one can find the equivalent statement
\begin{align}
\beta(\beta-p)(\varphi(0) - \varphi(p))(\varphi(0) - s) &+ \mu \cdot \beta(\varphi(0) - \varphi(p))(\varphi(\beta) - \varphi(p)) &\nonumber \\
-\beta p(\varphi(0) - \varphi(p))(\varphi(0) - s) &- \mu \cdot \beta p(\varphi(0) - \varphi(p))\varphi'(p) &\nonumber \\
-p^2(\varphi(\beta) - \varphi(0))(\varphi(0) - s) & - \mu \cdot p^2(\varphi(\beta) - \varphi(0))\varphi'(p) &\nonumber \\
p(\beta-p)\beta\varphi'(p)(\varphi(0) - s) & + \mu \cdot \beta p(\varphi(\beta) - \varphi(p))\varphi'(p) & \geq 0.
\end{align}
We next use \eqref{eq:upsilon_2} and substitute $s - \varphi(0)$ by $\mu \cdot (\varphi(\tau_2) - \varphi(0))/\tau_2$. As a result, each of the eight terms contains the factor $\mu$, so we can divide that out (as it is strictly positive). This results in the equivalent inequality
\begin{align}
\beta(\beta-p)(\varphi(p) - \varphi(0))\frac{\varphi(\tau_2)-\varphi(0)}{\tau_2} &+ \beta((\varphi(0) - \varphi(p))(\varphi(\beta) - \varphi(p)) & \nonumber \\
-\beta p(\varphi(p) - \varphi(0))\frac{\varphi(\tau_2)-\varphi(0)}{\tau_2} &- \beta p(\varphi(0) - \varphi(p))\varphi'(p) & \nonumber \\
+p^2(\varphi(\beta) - \varphi(0))\frac{\varphi(\tau_2)-\varphi(0)}{\tau_2} & -   p^2(\varphi(\beta) - \varphi(0))\varphi'(p) & \nonumber \\
-p(\beta-p)\beta\varphi'(p)\frac{\varphi(\tau_2)-\varphi(0)}{\tau_2} & +  \beta p(\varphi(\beta) - \varphi(p))\varphi'(p) & \geq 0.
\end{align}
We now multiply and divide every term by either $p,\beta$ or $(\beta-p)$ in order to make every term contain a factor of the form $\beta p(\beta-p)$ or $\beta p^2$. That is, the above inequality is equivalent to
\begin{align}
\beta p(\beta-p)\left[\frac{\varphi(p) - \varphi(0)}{p}\right]\left[\frac{\varphi(\tau_2)-\varphi(0)}{\tau_2}\right] &-  \beta p(\beta-p)\left[\frac{\varphi(p) - \varphi(0)}{p}\right]\left[\frac{\varphi(\beta) - \varphi(p)}{\beta-p}\right] & \nonumber \\[1.5ex]
-\beta p^2\left[\frac{\varphi(p) - \varphi(0)}{p}\right]\left[\frac{\varphi(\tau_2)-\varphi(0)}{\tau_2}\right] &+ \beta p^2\left[\frac{\varphi(p) - \varphi(0)}{p}\right]\varphi'(p) & \nonumber \\[1.5ex]
+\beta p^2\left[\frac{\varphi(\beta) - \varphi(0)}{\beta}\right]\left[\frac{\varphi(\tau_2)-\varphi(0)}{\tau_2}\right] & -   \beta p^2\left[\frac{\varphi(\beta) - \varphi(0)}{\beta}\right]\varphi'(p) & \nonumber\\[1.5ex]
-\beta p(\beta-p)\varphi'(p)\left[\frac{\varphi(\tau_2)-\varphi(0)}{\tau_2}\right] & +  \beta p(\beta-p)\left[\frac{\varphi(\beta) - \varphi(p)}{\beta-p}\right]\varphi'(p) & \geq 0.
\end{align}
We now order the four terms with factor $\beta p(\beta-p)$ together, and the four terms with factor $\beta p^2$, in order to get the equivalent inequality 
\begin{align}
\beta p^2\left[ \left( \frac{\varphi(\beta) - \varphi(0)}{\beta} - \frac{\varphi(p) - \varphi(0)}{p}\right) \left(\frac{\varphi(\tau_2) - \varphi(0)}{\tau_2} - \varphi'(p) \right)\right] & + \nonumber \\[1.5ex]
\beta p(\beta-p)\left[ \left( \varphi'(p) - \frac{\varphi(p) - \varphi(0)}{p}\right) \left(\frac{\varphi(\beta) - \varphi(p)}{\beta-p} - \frac{\varphi(\tau_2) - \varphi(0)}{\tau_2} \right)\right] & \geq 0.
\label{eq:quotient_rule_equiv1}
\end{align}
Convexity yields that 
$$
\left( \frac{\varphi(\beta) - \varphi(0)}{\beta} - \frac{\varphi(p) - \varphi(0)}{p}\right),\; \left( \varphi'(p) - \frac{\varphi(p) - \varphi(0)}{p}\right), \; \left(\frac{\varphi(\beta) - \varphi(p)}{\beta-p} - \frac{\varphi(\tau_2) - \varphi(0)}{\tau_2} \right) \geq 0,
$$
However, this does not necessarily hold for $$\left(\frac{\varphi(\tau_2) - \varphi(0)}{\tau_2} - \varphi'(p) \right) = F.$$ We now distinguish between two cases.\\
\noindent \textbf{Case 1: $F \geq 0$.} In this case, all factors of \eqref{eq:quotient_rule_equiv1} are non-negative, which means \eqref{eq:quotient_rule_equiv1} is true.\\
\noindent \textbf{Case 2: $F < 0$.} We can now rearrange terms such that all factors are non-negative and obtain equivalent inequality
\begin{align}
\beta p(\beta-p) \left( \varphi'(p) - \frac{\varphi(p) - \varphi(0)}{p}\right) \left(\frac{\varphi(\beta) - \varphi(p)}{\beta-p} - \frac{\varphi(\tau_2) - \varphi(0)}{\tau_2} \right) & \geq \nonumber \\
\beta p^2 \left( \frac{\varphi(\beta) - \varphi(0)}{\beta} - \frac{\varphi(p) - \varphi(0)}{p}\right) \left(\varphi'(p) - \frac{\varphi(\tau_2) - \varphi(0)}{\tau_2} \right)&.
\label{eq:quotient_rule_equiv3}
\end{align}
\noindent Now convexity yields that
$$\left( \varphi'(p) - \frac{\varphi(p) - \varphi(0)}{p}\right) \geq \left(\varphi'(p) - \frac{\varphi(\tau_2) - \varphi(0)}{\tau_2} \right),$$
so that \eqref{eq:quotient_rule_equiv3} is true if the following inequality holds:
\begin{align}
\beta p(\beta-p) \left(\frac{\varphi(\beta) - \varphi(p)}{\beta-p} - \frac{\varphi(\tau_2) - \varphi(0)}{\tau_2} \right) & \geq \nonumber \\
\beta p^2 \left( \frac{\varphi(\beta) - \varphi(0)}{\beta} - \frac{\varphi(p) - \varphi(0)}{p}\right) &.
\label{eq:quotient_rule_equiv4}
\end{align}
After expanding the brackets on both sides and making some rearrangements, we obtain
\begin{align}
\frac{\varphi(\beta)-\varphi(0)}{\beta} \geq \frac{\varphi(\tau_2)-\varphi(0)}{\tau_2},
\label{eq:quotient_rule_equiv5}
\end{align}
which holds due to convexity.

\section{Proof of Theorem \ref{th:optimal_price_variance}}\label{appendix_variance}
We will establish the proof with two lemmas.
\begin{lemma}\label{lemma:price1}
    The optimal robust price $p^*_{(\mu,\sigma,\beta)} \in \{p^*_l,p^*_h\}$.
\end{lemma}
\begin{proof}
Notice that $Y = X/\mu$ satisfies the system
\begin{align}
    \left\{
\begin{array}{l}
     \E[1] = 1 \\
     \E[Y] = 1\\
     \E[Y^2] = (\mu^2+\sigma^2)/\mu^2 \\
     0 \leq Y \leq \beta/\mu
\end{array}\right..
\label{eq:equations_scaled}
\end{align}
This means that without loss of generality, we may assume that $\mu = 1$. For ease of analysis, we define the functions
$$f_{1a}(p) = \frac{(1-p)^2}{(1-p)^2+\sigma^2},\ \ f_{1b}(p) = \frac{p(1-p)}{(1-p)+\sigma^2},\ \ f_{2a}(p) = \frac{p(1+\sigma^2-p)}{(\beta-p)(\beta+p-1-\sigma^2)}, \ \ f_{2b}(p) = \frac{p}{\beta}.$$

First, consider the case that $p_{(\mu,\sigma,\beta)}^* \in (0,\tau_1]$. This analysis corresponds to the case of $\beta = \infty$, as $R(p,\sigma)$ does not depend on $\beta$, although this is performed over the smaller interval $(0,\tau_1]$. Over the larger interval $(0, \mu)$, \cite{giannakopoulos2019robust} proves that $p_l^*$ is the unique global maximum, $f_{1b}(p) < f_{1a}(p)$ while $p \in (0,p^*_l]$, and $f_{1b}(p)$ is increasing in $p$. Hence, over the interval $(0,\tau_1]$, the local maximum location is $\min\{p_l^*,\tau_1\}$. 

Second, consider the case that $p_{(\mu,\sigma,\beta)}^* \in [\tau_1,\tau_2)$.  Observe that
\begin{align}
    f'_{2a}(p) = \frac{\beta(\sigma^2-\beta+1)(2p-\sigma^2-1)}{(\beta-p)^2(p-\sigma^2+\beta-1)^2}\geq 0
\end{align}
if and only if 
\begin{align}\label{eq:a}
    (\sigma^2-\beta+1)(2p-\sigma^2-1) \geq 0.
\end{align} Since $\beta > \tau_2 = 1 + \sigma^2 > 0$, we know \eqref{eq:a} holds if and only if $\tau_2 - 2p \geq 0$.
In case $\tau_1 < \tau_2/2$, this then proves that $f'_{2a}(p)$ has at most one root (located at $\tau_2/2$), is non-negative on the interval $[\tau_1,\tau_2/2]$, and is non-positive in $p$ on the interval $[\tau_2/2,\tau_2]$. This in turn means that $f_{2a}(p)$ has a global maximum location at $\tau_2/2$, is increasing in $p$ on $[\tau_1,\tau_2/2]$, and decreasing in $p$ on the interval $[\tau_2/2,\tau_2]$. In case $\tau_1 \geq \tau_2/2$, then $f_{2a}(p)$ is non-positive in $p$ and $p^*_h \neq \tau_2/2$.

Furthermore, it is immediately clear that $f_{2b}(p)$ is strictly increasing in $p$. We will now consider the intersection points of $f_{2a}(p)$ and $f_{2b}(p)$. Notice that $$f_{2a}(p) - f_{2b}(p) = \frac{\beta p(1+\sigma^2-p)-p(\beta-p)(p-1-\sigma^2+\beta)}{\beta(\beta-p)(p-1-\sigma^2+\beta)}$$ is a rational function with as numerator a third-degree polynomial function. Hence, there will be at most three real intersection points. One intersection point is 0 and is not relevant. Additionally, the intersection point $\frac{1}{2}\left(\tau_2+\beta+\sqrt{(3\beta-\tau_2)^2-4\beta^2}\right) > \tau_2$ and is therefore also irrelevant. The intersection point $$r_h = \frac{1}{2}\left(\tau_2+\beta-\sqrt{(3\beta-\tau_2)^2-4\beta^2}\right)$$ can lie in the interval $[\tau_1,\tau_2)$ and is therefore relevant. Also, one can show that $r_h < \tau_2$ if and only if $\tau_2 < \beta$. Hence, we can write the local maximum location over the interval $[\tau_1,\tau_2)$ as $\max\{\tau_1,p_{h}^*\}$.

Now for $\tau_1$ to be the global maximum that differs from $p^*_{l}$ and $p^*_{h}$, the following ordering has to hold:
\begin{align*}
    p^*_{h} < \tau_1 < p^*_{l}.
\end{align*}
Now assume $p^*_{h} < \tau_1 < p^*_{l}$. We then know that that $f_{1a}(\tau_1) > f_{1b}(\tau_1)$, as the intersection of $f_{1a}$ and $f_{1b}$ lies to the right of $\tau_1$. Additionally, we can infer that $f_{2a}(p)$ is decreasing in $p$ over the interval $[\tau_1,\tau_2]$, as $\tau_1 > p^*_h \geq \tau_2/2$. Now since $f_{2b}(p)$ is increasing in $p$, we must have that $f_{2a}(\tau_1)<f_{2b}(\tau_1)$, as otherwise $\tau_1 \leq r_h \leq p^*_h$. However, observe that $f_{1a}(\tau_1) = f_{2a}(\tau_1)$ and $f_{1b}(\tau_1) = f_{2b}(\tau_1)$, which is a contradiction.
\end{proof}

The next lemma shows that $p^* = p^*_l$ when the variance is low and $p^* = p^*_h$ when the variance is high.

\begin{lemma}
    Consider $\sigma^*$ from \textup{Theorem \ref{th:optimal_price_variance}}. When $\sigma \leq \sigma^*$, then $p^*_{(\mu,\sigma,\beta)} = p^*_l$ and when $\sigma \geq \sigma^*$, then $p^*_{(\mu,\sigma,\beta)} = p^*_h$.
\end{lemma}
\begin{proof}
    As motivated in the proof of Lemma \ref{lemma:price1}, we will assume without loss of generality that $\mu=1$. Furthermore, we recall that $$f_{1a}(p) = \frac{(1-p)^2}{(1-p)^2+\sigma^2},\ \ f_{1b}(p) = \frac{p(1-p)}{(1-p)+\sigma^2},\ \ f_{2a}(p) = \frac{p(1+\sigma^2-p)}{(\beta-p)(\beta+p-1-\sigma^2)}, \ \ f_{2b}(p) = \frac{p}{\beta}.$$

    Now consider the case that $\sigma^2 \xrightarrow{} \mu(\beta-\mu)$, which is the maximal value of $\sigma^2$. Then $p^*_{(\mu,\sigma,\beta)} = p^*_h$, as $\tau_1 \xrightarrow{} 0$ implies $p^*_l \xrightarrow{} 0$. Next, consider the case that $\sigma^2 = 0$. Then $p^*_{(\mu,\sigma,\beta)} = p^*_l$, as $p^*_l = 1$ with $f_{1a}(p^*_l) = f_{1b}(p^*_l) = 1$, which is the highest value for the competitive ratio, while $p^*_h$ performs strictly worse.
    
    Since $f_{1a}(p)$ and $f_{1b}(p)$ are continuously differentiable in $p^*_l$, we can use the envelope theorem to show that
    \begin{align}
        \frac{\partial f_{1a}(p^*_l)}{\partial \sigma^2} = \frac{\partial f_{1a}(p)}{\partial\sigma^2}\biggr\rvert_{p = p^*_l} = -\frac{\left(1 - p^*_l\right)^{2}}{\left(\sigma^2 + \left(1 - p^*_l\right)^{2}\right)^{2}} \leq 0,
    \end{align}
    and
    \begin{align}
        \frac{\partial f_{1b}(p^*_l)}{\partial \sigma^2} = \frac{\partial f_{1b}(p)}{\partial\sigma^2}\biggr\rvert_{p = p^*_l} = \frac{\left(p^*_l - 1\right) p^*_l}{\left(\sigma^2 - p^*_l + 1\right)^{2}} \leq 0,
    \end{align}
    as $p^*_l \leq 1$. Hence, $\min\{f_{1a}(p^*_l),f_{1b}(p^*_l)\}$ is decreasing in $\sigma^2$. Furthermore, since $f_{2a}(p)$ and $f_{2b}(p)$ are continuously differentiable in $p^*_h$, we can use the envelope theorem to show that
    \begin{align}
        \frac{\partial f_{2a}(p^*_h)}{\partial \sigma^2} = \frac{\partial f_{2a}(p)}{\partial\sigma^2}\biggr\rvert_{p = p^*_h} = \frac{\beta p^*_h}{(\beta-p^*_h)(\sigma^2-p^*_h-\beta+1)^2} \geq 0,
    \end{align}
    and
    \begin{align}
        \frac{\partial f_{2b}(p^*_h)}{\partial \sigma^2} = \frac{\partial f_{2b}(p)}{\partial\sigma^2}\biggr\rvert_{p = p^*_h} = 0 \geq 0.
    \end{align}
    Hence, $\min\{f_{2a}(p^*_h),f_{2b}(p^*_h)\}$ is increasing in $\sigma^2$. This finishes the proof.
\end{proof}

\section{Proof of Lemma \ref{lemma:ulm}}\label{appendix_powermoments}
We start with proving that $g_{1b}(p)$ with $p \in (0,\tau_1]$ has one local maximum location. Taking the derivative yields 
\begin{align}\label{eq:z1}
    g_{1b}'(p) = \frac{\alpha(p) - p\alpha'(p)}{\alpha(p)^2}.
\end{align}
We know that $\alpha(p) \in [\mu,\beta]$, so that $\alpha(p)^2 \neq 0$. Hence, since we are only interested in the sign of $g_{1b}'(p)$, we can continue the analysis with $\alpha(p) - p\alpha'(p)$. Recall from \eqref{alpha} that $\alpha(p)$ is the unique solution to
\begin{align}\label{alpha_moment}
    \frac{sp-\mu p^q+\mu \alpha(p)^q-s \alpha(p)}{p \alpha(p)^q-\alpha(p)p^q} = 1.
\end{align}
Therefore, we can express $\alpha'(p)$ in terms of $\alpha(p)$ and $p$ by differentiating \eqref{alpha_moment} with respect to $p$ to obtain
\begin{align}
    &\frac{{\mu} \left(q\alpha(p)^{q-1} \alpha'(p) - qp^{q-1}\right) - s \left(\alpha'(p) - 1\right)}{p \alpha(p)^q - p^q \alpha(p)}\nonumber \\
    &- \frac{\left({\mu} \left(\alpha(p)^q - p^q\right) - s \left(\alpha(p) - p\right)\right) \left(pq\alpha(p)^{q-1} \alpha'(p) - p^q \alpha'(p) + \alpha(p)^q - qp^{q-1} \alpha(p)\right)}{\left(p \alpha(p)^q - p^q \alpha(p)\right)^{2}} = 0.
\end{align}

\noindent Clearly, $p \alpha(p)^q - p^q \alpha(p) \neq 0$ due to \eqref{alpha_moment} having a solution on $(0,\tau_1)$, so we can solve for $\alpha'(p)$ to obtain
\begin{align}
    \alpha'(p) = \frac{\left(qp^{q-1} \left(p - \alpha(p)\right) - p^{q} + \alpha(p)^{q}\right) \left(\alpha(p)^{q} \mu - \alpha(p)s\right)}{\left(q\alpha(p)^{q-1} \left(p - \alpha(p)\right) + \alpha(p)^{q} - p^{q}\right) \left(p^{q} \mu - ps\right)}.\label{eq:alphaprime}
\end{align}

\noindent Plugging \eqref{eq:alphaprime} into the numerator of $g'_{1b}(p)$ and simplifying results in 
\begin{align}
    \frac{\left(p \alpha(p)^{q} - p^{q} \alpha(p)\right) \left(qs \left(\alpha(p) - p\right)-\mu \left(\alpha(p)^{q} - p^{q}\right)\right)}{\left(sp-\mu p^{q}\right) \left(q \alpha(p)^{q-1} \left(\alpha(p) - p\right)- \left(\alpha(p)^{q} - p^{q}\right)\right)} = \frac{F_1F_2}{F_3F_4}.
\end{align}
\noindent Now observe that $F_1>0$, as $\alpha(p)>p$, $F_3>0$, as $\tau_2>p$, and $F_4>0$, as strict convexity yields $q\alpha(p)^{q-1} > \frac{\alpha(p)^q-p^q}{\alpha(p)-p}$. Hence, since we are only interested in the sign of $g'_{1b}(p)$, we can continue the analysis with $$qs \left(\alpha(p) - p\right)-\mu \left(\alpha(p)^{q} - p^{q}\right).$$
One can show that
\begin{align}\label{p_max_2}
qs \left(\alpha(p) - p\right)-\mu \left(\alpha(p)^{q} - p^{q}\right) \geq 0 \iff
\frac{\alpha(p)^q-p^q}{\alpha(p)-p}\leq \frac{qs}{\mu}.
\end{align}
Notice that the left-hand side of \eqref{p_max_2} is increasing in $p$ when $\alpha(p)$ is constant. Additionally, it is increasing in $\alpha(p)$ when $p$ is constant. Now since $\alpha(p)$ is increasing in $p$, it also holds that the left-hand side is increasing in $p$ overall. Simultaneously, the right-hand side is constant in $p$. This means there is one local maximum location.

We now continue with showing that $g_{2a}(p)$ with $p \in [\tau_1,\tau_2]$ has at most one local maximum location. Taking the derivative results in
\begin{align}
    g'_{2a}(p) = \frac{\left(\beta^{q} \mu - \beta s\right) \left(sp-q\mu p^{q}\right)}{p \left(\mu \left( \beta^{q}-p^q\right) - s \left( \beta-p\right)\right)^{2}}.
\end{align}
\noindent Furthermore, $\beta^q\mu - \beta s \geq 0$, as $\beta \geq \tau_2 = (\frac{s}{\mu})^{\frac{1}{q-1}}$. Hence, since we are only interested in the sign of $g_{2a}'(p)$, we can continue the analysis with $sp - q\mu p^q.$ It holds that
$$sp - q\mu p^q>0 \iff p < \left(\frac{s}{q\mu}\right)^{\frac{1}{q-1}}.$$
Since the left-hand side is increasing in $p$ and the right-hand side is constant, there is one local maximum location.

\end{document}